\documentclass[reqno]{amsart}

\usepackage{graphics}
\usepackage{latexsym}
\usepackage{amsfonts}
\usepackage{amsmath}
\usepackage{amssymb}
\usepackage{amsthm}
\usepackage{multicol}
\usepackage{dsfont}
\usepackage{enumerate}

\numberwithin{equation}{section}

\newtheorem{theorem}{Theorem}[section] 
\newtheorem{proposition}{Proposition}[section] 
\newtheorem{lemma}{Lemma}[section] 
\newtheorem{remark}{Remark}[section] 

\newcommand{\R}{\mathbb R} 
\newcommand{\C}{\mathbb C} 

\renewcommand{\P}{\mathbb P}
\newcommand{\E}{\mathbb E}
\newcommand{\V}{\mathbb V}
\newcommand{\Cov}{\mathbb Cov}

\newcommand{\tr}{\text{tr}}
\newcommand{\Tr}{\text{Tr}}

\renewcommand\Re{\operatorname{\mathfrak{Re}}}
\renewcommand\Im{\operatorname{\mathfrak{Im}}}

\begin{document}
\title[Fluctuations of Matrix Entries]{Fluctuations of Matrix Entries of Regular Functions of Wigner Matrices}
\author[A. Pizzo]{Alessandro Pizzo}
\address{Department of Mathematics, University of California, Davis, One Shields Avenue, Davis, CA 95616-8633  }
\thanks{A.P. has been supported in part by the NSF grant DMS-0905988}
\email{pizzo@math.ucdavis.edu}

\author[D. Renfrew]{David Renfrew} \thanks{D.R. has been supported in part by the NSF grants VIGRE DMS-0636297, DMS-1007558, and DMS-0905988 }
\address{Department of Mathematics, University of California, Davis, One Shields Avenue, Davis, CA 95616-8633  }
\email{drenfrew@math.ucdavis.edu}

\author[A. Soshnikov]{Alexander Soshnikov}
\address{Department of Mathematics, University of California, Davis, One Shields Avenue, Davis, CA 95616-8633  }
\thanks{A.S. has been supported in part by the NSF grant DMS-1007558}
\email{soshniko@math.ucdavis.edu}

\begin{abstract}
We study the fluctuations of the matrix entries of regular functions of Wigner random matrices in the limit when the 
matrix size goes to infinity.  In the case of the 
Gaussian ensembles (GOE and GUE) this problem was considered by A.Lytova and L.Pastur in \cite{LP}.  Our results are valid provided the 
off-diagonal matrix entries have finite fourth moment, the diagonal matrix entries have finite second moment, and the 
test functions have four continuous derivatives in a neighborhood of the support of the Wigner semicircle law.
Moreover, if the marginal distributions satisfy the Poincar\'e inequality our results are valid for Lipschitz continuous test functions.
\end{abstract}



\maketitle

\section{ \bf{Introduction}}
\label{sec:intro}

Let $X_N= \frac{1}{\sqrt{N}} W_N $ be a random Wigner real symmetric (Hermitian) matrix.
In the real symmetric case, we assume that the off-diagonal entries 
\begin{equation}
\label{offdiagreal1}
(W_N)_{jk},\ 1\leq j<k \leq N,
\end{equation}
are i.i.d. random variables with probability distribution $\mu,$ such that
\begin{equation}
\label{offdiagreal2}
\E (W_N)_{jk}=0, \ \V(W_N)_{jk}=\sigma^2, \ \E (W_N)_{jk}^4=m_4<\infty, \ 1\leq j<k \leq N,
\end{equation}
where $\E \xi $ denotes the mathematical expectation and $\V \xi $ the variance of a random variable $\xi.$
The diagonal entries  
\begin{equation}
\label{diagreal1}
(W_N)_{ii}, \ 1\leq i \leq N,
\end{equation}
are i.i.d. random variables, independent from the off-diagonal entries, such that
\begin{equation}
\label{diagreal2}
\E (W_N)_{ii}=0, \ \V(W_N)_{ii}=\sigma_1^2, \ 1\leq i \leq N.
\end{equation}
We will denote the probability distribution of $\frac{1}{\sqrt{2}}\*(W_N)_{11}$ by $\mu_1. $

In a similar fashion, in the Hermitian case, we assume that the off-diagonal entries
\begin{equation}
\label{offdiagherm1}
\sqrt{2} \Re (W_N)_{jk}, \sqrt{2} \Im (W_N)_{jk}, \ 1\leq j<k \leq N,
\end{equation}
are i.i.d. centered random variables with probability distribution $\mu$ with variance $\sigma^2$ and finite fourth moment $m_4. $ 
The diagonal entries  
\begin{equation}
\label{diagherm1}
(W_N)_{ii}, \ 1\leq i \leq N,
\end{equation}
are i.i.d. random variables, independent from the off-diagonal entries, with probability distribution $\mu_1$ and finite second moment.

While the independence of the matrix entries $(W_N)_{ij}, \ 1\leq i\leq j\leq N$ is crucial in our analysis, the requirement that the entries are 
identically distributed can be replaced by certain Lindeberg-Feller type conditions for the fourth moments of marginal distributions (\cite{ORS}).

Given a  real symmetric (Hermitian) matrix $B$ of order $N,$ we define its empirical distribution of the eigenvalues as 
$\mu_B = \frac{1}{N} \sum_{i=1}^{N} \delta_{\lambda_{i}},$ where  $\lambda_1 \leq \ldots \leq \lambda_N$ are the (ordered) eigenvalues of $B.$
One of the fundamental results of random matrix theory is the celebrated Wigner semicircle law (see e.g. \cite{BG}, \cite{AGZ}, \cite{B}).  It states 
that almost surely $\mu_{X_N}$ converges weakly to the nonrandom limiting distribution $\mu_{sc}$ whose density is 
given by 
\begin{equation}
\label{polukrug}
\frac{d \mu_{sc}}{dx}(x) = \frac{1}{2 \pi \sigma^2} \sqrt{ 4 \sigma^2 - x^2} \mathbf{1}_{[-2 \sigma , 2 \sigma]}(x),
\end{equation}
In other words, for any bounded continuous test function $\varphi: \R \to \R,$ the linear statistic 
\[ \frac{1}{N} \sum_{i=1}^N \varphi(\lambda_i) = \frac{1}{N} \*\Tr(\varphi(X_N))=:\tr_N(\varphi(X_N)) \]
converges to $\int \varphi(x) \* d \mu_{sc}(dx) $ almost surely; here and throughout the paper we 
use the notation $\tr_N = \frac{1}{N} \Tr$ to denote the normalized trace.

The Stieltjes transform of the semi-circle law is
\begin{equation}
\label{steltsem} 
g_\sigma(z) := \int \frac{d \mu_{sc}(x)}{z-x}= \frac{z-\sqrt{z^2-4\*\sigma^2}}{2\*\sigma^2}, \ z \in \C \backslash [-2\*\sigma, 2\*\sigma].
\end{equation}
It is the solution to 
\begin{equation}
\label{semicircle}
\sigma^2 g_\sigma^2(z) - z g_\sigma(z) + 1 = 0
\end{equation}
that decays to 0 as $|z| \to \infty$. 

In this paper, we are interested in studying the joint distribution of matrix entries of regular functions of a Wigner random matrix $X_N.$
In \cite{LP}, Lytova and Pastur studied the limit of the one dimensional distribution of $\sqrt{N} \* \left( f(X_N)_{ij} - \E (f(X_N)_{ij})\right) $
in the case of GOE(GUE) ensembles (so the marginal distribution $\mu$ of matrix entries is Gaussian) provided $f(x) $ is a 
bounded differentiable function with bounded derivative.  Namely, they prove that
\begin{equation}
\label{roscha}
\sqrt{N} \* \left( f(X_N)_{ij} - \E (f(X_N)_{ij})\right) \to N(0, \frac{1+\delta_{ij}}{\beta}\* \omega^2(f)),
\end{equation}
with $\beta=1(2) $ in the GOE (GUE) case,
\begin{equation}
\label{omegasq}
\omega^2(f):= \V(f(\eta))= \frac{1}{2}\* \int_{-2\sigma}^{2\sigma} \int_{-2\sigma}^{2\sigma} (f(x)-f(y))^2 
\*  \frac{1}{4 \pi^2 \sigma^4} \sqrt{ 4 \sigma^2 - x^2} \* 
\sqrt{ 4 \sigma^2 - y^2} \* dx \* dy
\end{equation}
where $\eta$ is distributed according to the Wigner semicircle law (\ref{polukrug}). In the case of the off-diagonal entries in the GUE case,
the r.h.s. in (\ref{roscha}) should be understood as a complex Gaussian distribution with independent identically distributed real and imaginary 
parts, 
each with the variance $\frac{1}{2} \*\omega^2(f). \ $ The proof in \cite{LP} relies on the orthogonal (unitary)
invariance of the GOE (GUE) ensembles.

We extend the results of \cite{LP} in the following way.
We study the joint distribution of any finite number of the the matrix entries $f(X_N)_{ij}$ 
The limiting distribution
on the r.h.s. of (\ref{roscha}) is, in general, no longer Gaussian.  Instead, it is
the distribution of a linear combination of $(W_N)_{ij}$ and a Gaussian random variable, independent from $(W_N)_{ij}$
(see Theorems \ref{thm:real} and \ref{thm:herm} below). We refer the reader to Remark \ref{ZAM} after Theorem \ref{thm:real} for the discussion on 
when one of the two components in the linear combination vanishes.  In particular, the limiting distribution of 
$\sqrt{N} \* \left( f(X_N)_{ij} - \E (f(X_N)_{ij})\right)$ is Gaussian if and only if either the marginal distribution is Gaussian or 
$\int_{-2\sigma}^{2\sigma} x\*f(x)\*\sqrt{4\sigma^2-x^2}\* dx =0.$

Our approach requires that $f$ has four continuous derivatives in a neighborhood of the support of the Wigner semicircle law, 
$[-2\*\sigma, 2\*\sigma]. $ 
If the marginal distributions $\mu $ and $\mu_1$ satisfy a Poincar\'e inequality (\ref{poin}) then our results hold provided $f$ is 
Lipschitz continuous in a neighborhood of 
$[-2\*\sigma, 2\*\sigma]. $ 

The problem about the fluctuation of the entries of $f(X_N)$ is interesting in its own right.  However, for us the main motivation to study the 
problem came from the question about the limiting distribution of the outliers of finite rank perturbations of standard Wigner matrices 
(see e.g. \cite{CDF}, \cite{CDF1}, \cite{PRS}).  Let $M_N=X_N+C_N, $ where $X_N$ is a random real symmetric (Hermitian) Wigner matrix defined above 
and $C_N$ is a deterministic
real symmetric (Hermitian) matrix of finite rank $k$ with fixed non-zero eigenvalues $\lambda_1(C_N), \ldots, \lambda_k(C_N)$ and the corresponding 
orthonormal eigenvectors $v_1, \ldots, v_k.$   By the interlacing property, $M_N$ has at most $k$ eigenvalues (called outliers) 
that stay outside the support of the Wigner semicircle law in the limit of large $N.$  
Capitaine, Donati-Martin, and F\'eral proved in \cite{CDF1} that the limiting distribution of the 
outliers depends on the localization/delocalization properties of the eigenvectors $v_1, \ldots, v_k.$  In particular, if the eigenvectors are localized
(so only a finite number of coordinates are non-zero as $N \to \infty$), then the limiting distribution of the outliers is non-universal and depends on the 
marginal distribution of the matrix entries of $W_N.$  

The results in \cite{CDF1} are proved under the assumption that the marginal distribution of the 
i.i.d. entries of $W_N$ is symmetric and satisfies the Poincar\'e inequality (\ref{poin}).  In \cite{PRS}, we have extended the results of \cite{CDF1} to
the case of a finite fifth moment. Our approach relies on an ideas from \cite{BGM}.  In particular, an important step 
of the proof is the study of the limiting distribution of the eigenvalues of the $k\times k$ matrix 
$( \langle v_i, R_N(z) v_j \rangle)_{1\leq i,j\leq k},$ where $R_N(z)$ is the resolvent of $X_N$ and  
$\langle \cdot, \cdot \rangle$ is the standard inner product in $\C^N$ (see Proposition 1 in \cite{PRS}).
Thus, in the localized case one is interested in the joint distribution of a finite number of
resolvent entries of a standard Wigner matrix.

The rest of this section is devoted to the explanation of the main idea of the proof.  Complete
formulations of the results are given in Section \ref{sec:formulation}.
We restrict our attention to the real symmetric case since the arguments in the Hermitian case are very similar.
 
We start by considering the test functions of the form  $f(x)=\frac{1}{z-x}$ which corresponds to studying the resolvent entries. Define 
\begin{equation*}
R_N(z):=(z\*I_N-X_N)^{-1}, 
\end{equation*}
the resolvent of a real symmetric Wigner matrix $X_N=\frac{1}{\sqrt{N}}W_N$ for $z$ outside the spectrum of $X_N.$ 
For simplicity, we will consider here a diagonal 
entry of $R_N(z),$ say the $(N,N)-$th entry $(R_N(z))_{NN}.$  The off-diagonal entries can be treated in a similar way 
(see Section \ref{sec:flucmat}).  When it does not 
lead to ambiguity, we will use the shorthand notations $R_{ij}$ for $(R_N(z))_{ij}$ and $W_{ij}$ for $(W_N)_{ij}.$  
Further, assume that $z$ is fixed and $\Im z \not=0.$  By Cramer's rule, the $(N,N)-$th  entry of the resolvent can be written as
\begin{equation}
\label{ireland}
R_{NN}=(z-\frac{1}{\sqrt{N}}\*W_{NN} - \frac{1}{N}\*b^t \*G \*b)^{-1},
\end{equation}
where $b$ is the $(N-1)-$ dimensional column vector with the coordinates $W_{i1}, \ 1\leq i\leq N-1, \ b^t$ its transpose, and 
$G=G(z)$  is the resolvent of the $(N-1)\times (N-1)$ upper-left submatrix 
$\tilde{X}=(\frac{1}{\sqrt{N}}\*W_{ij})_{1\leq i,j\leq N-1}$ of the Wigner matrix $X_N,$ i.e.
\begin{equation*}
G:=(z\*I_{N-1} -\tilde{X})^{-1}.
\end{equation*}
We note that $\tilde{X}$ can be viewed as a standard $(N-1)\times(N-1)$ real symmetric Wigner matrix since the normalization by $\sqrt{N}$ instead of 
$\sqrt{N-1}$ does not make any difference in the limit of large $N.$
It is very important in our analysis that the random variables $W_{NN}$ and $b^t \*G\*b$ in (\ref{ireland}) are independent.  Moreover, in the 
quadratic form 
\begin{equation}
\label{macedonia}
b^t \*G\*b=\sum_{1\leq i,j\leq N-1} G_{ij} \*W_{i1}\*W_{j1},
\end{equation}
the vector $b$ is independent from the matrix $\tilde{R}.$  By subtracting and adding  $\E (\frac{1}{N}\*b^t \*G\*b)= 
\sigma^2 \*\E (\frac{1}{N}\*\Tr G)$ in the denominator,
 we rewrite (\ref{ireland}) as
\begin{align}
\label{ireland1}
& R_{NN}=\left(z- \E (\frac{1}{N}\*b^t \*G\*b)- 
\frac{1}{\sqrt{N}}\*W_{NN} - (\frac{1}{N}\*b^t \*G\*b - \E (\frac{1}{N}\*b^t \*G\*b))\right)^{-1} \\
\label{ireland2}
& = \left(z- \sigma^2 \*\E (\frac{1}{N}\*\Tr G)- 
\frac{1}{\sqrt{N}}\*W_{NN} - (\frac{1}{N}\*b^t \*G\*b - \frac{\sigma^2}{N} \*\E (\Tr G))\right)^{-1} \\
\label{ireland3}
& =\left(z- \sigma^2\*\E (\frac{1}{N}\*\Tr R_N(z)) +O(\frac{1}{N}) - 
\frac{1}{\sqrt{N}}\*W_{NN} - (\frac{1}{N}\*b^t \*G\*b - \frac{\sigma^2}{N} \*\E (\Tr G))\right)^{-1},
\end{align}
where in (\ref{ireland3}) we used the interlacing property satisfied by the eigenvalues of $X_N$ and its submatrix $\tilde{X}$
to write
\begin{equation*}
\E (\frac{1}{N}\*\Tr G)= \E (\frac{1}{N}\*\Tr R_N(z))+O(\frac{1}{N})
\end{equation*}
since $\Im z \not=0$ is fixed.

It follows from the semicircle law that
\begin{equation*}
\lim_{N\to \infty}  \E \frac{1}{N} \*\Tr R_N(z) =g_{\sigma}(z),
\end{equation*}
where the Stieltjes transform $g_{\sigma}(z)$ of the semicircle law has been defined in (\ref{steltsem}).
It follows from the calculations in Section \ref{sec:expvar} (see (\ref{odinnadtsat100})  in Proposition \ref{proposition:prop1}) that for fixed $z$
\begin{equation*}
\E \frac{1}{N}\* \Tr R_N(z) =g_{\sigma}(z) +O(\frac{1}{N}).
\end{equation*}
Therefore,
\begin{equation}
\label{ireland4}
 R_{NN}= \left(z- \sigma^2 \* g_{\sigma}(z) +O(\frac{1}{N}) - 
\frac{1}{\sqrt{N}} \left(\*W_{NN} + \frac{1}{\sqrt{N}} \*(b^t \*G\*b - \sigma^2\*\E (\Tr G))\right)\right)^{-1}.
\end{equation}
As we have already remarked, the random variables $W_{NN}$ and $\frac{1}{\sqrt{N}} \*\left(b^t \*G\*b - 
\sigma^2 \*\E (\Tr G)\right)$ are 
independent.  The crucial step in the analysis of the fluctuation of the resolvent entries is to prove that 
\begin{equation}
\label{sluchvel}
\frac{1}{\sqrt{N}} \*\left(b^t \*G\*b - \sigma^2 \*\E (\Tr G)\right)
\end{equation}
converges in distribution as $N \to \infty$ to a 
centralized complex Gaussian random variable. We discuss why 
such a convergence in distribution takes place a few paragraphs below but first we note that the Central Limit Theorem for
(\ref{sluchvel}) together with the formula (\ref{ireland4}) for $R_{NN}$ and
\begin{equation*}
z-\sigma^2\*g_{\sigma}(z)=\frac{1}{g_{\sigma}(z)}
\end{equation*}
from (\ref{semicircle})
immediately implies that the normalized resolvent entry
$\sqrt{N} (R_{NN}- g_{\sigma}(z))$ converges in distribution as $N \to \infty$ to the law
of a linear combination of $W_{NN}$ and a complex centralized Gaussian random variable.  
The coefficients of the linear combination and the covariance matrix of the complex 
Gaussian random variable are easily computable (see Theorem \ref{thm:resreal} in the next section).

In order to get insight into the limiting distribution of (\ref{sluchvel}), it is useful to consider first the case of a quadratic form with deterministic
coefficients. Let $b$ be a random $n-$dimensional vector with centralized i.i.d. real components with unit variance and finite fourth moment, 
and $A_n$ be an $n\times n$ deterministic real 
symmetric matrix such that
\begin{align}
\label{nonrandquad1}
& \text{the operator norm} \ \|A_n\| \leq a \ \text{for all} \ n\geq 1, \\
\label{nonrandquad2}
& \frac{1}{n} \* \Tr (A_n^2) \to a_2 \ \text{as} \ n\to \infty, \\
\label{nonrandquad3}
& \frac{1}{n} \* \sum_{i=1}^n (A_n)_{ii}^2 \to a_1^2  \ \text{as} \ n\to \infty,
\end{align}
where $a>0$ is some constant that does not depend on $n.$ 
The CLT for $\frac{1}{\sqrt{n}} \*\left(b^t \*A_n\*b - \sigma^2 \*\E (\Tr A_n)\right)$ was first established by Sevast'yanov (\cite{Sev}) in the case when 
the coordinates of $b$ are i.i.d. standard Gaussian random variables.  For subsequent developments, we refer the reader to
\cite{W}, \cite{Ber}, and \cite {BY}.  In particular, since
$$b^t \*A_n\*b - \sigma^2 \*\E (\Tr A_n)=\sum_{j=1}^n \left( a_{ii} \*(b_i^2-\sigma^2) +\sum_{i<j} a_{ij} \*b_i\*b_j\right)=\sum_{j=1}^n Z_n$$
we can write 
$b^t \*A_n\*b - \sigma^2 \*\E (\Tr A_n)$ as a sum of martingale differences with respect to the filtration 
$\mathcal{F}_j=\sigma(b_1, \ldots, b_j), \
j=1, \ldots, n.$  It is not difficult to prove that the conditions (\ref{nonrandquad1}-\ref{nonrandquad3}) 
imply that the Central Limit Theorem for martingale differences
(see e.g. \cite{D}) is applicable and 
the normalized random variable
$ \frac{1}{\sqrt{n}} \* (b^t \*A_n\*b - \Tr A_n)$ converges in distribution to $N(0, \kappa_4\*a_1^2 +a_2)$ as $n\to \infty,$ where $\kappa_4$ is 
the fourth cumulant of the marginal one-dimensional distribution of $b.$ 
In (\ref{sluchvel}), the quadratic form is associated with a complex symmetric random matrix $G.$  Thus, first of all, one has to study the 
joint distribution of 
\begin{equation*}
\frac{1}{\sqrt{N}} \*\left(b^t \* \Re G\*b - \sigma^2 \*\E (\Tr \Re G)\right) \ \text{and} \ 
\frac{1}{\sqrt{N}} \*\left(b^t \* \Im G\*b - \sigma^2 \*\E (\Tr \Im G)\right),
\end{equation*}
where
\begin{equation*}
\Re G=(\Re G_{ij}, \ 1\leq i,j\leq N-1), \ \ \Im G=(\Im G_{ij}, \ 1\leq i,j\leq N-1).
\end{equation*}
This corresponds to the choice of 
\begin{equation}
\label{annnnn}
A_n= x \* \Re G + y \* \Im G, \ n=N-1, 
\end{equation}
where $x$ and $y$ are arbitrary real numbers.  The second difference is that $A_n$ in (\ref{annnnn}) is random.  However, 
the CLT 
still holds if $\|A_n\|$ is bounded from above by a non-random constant $a$ with probability $1$ 
and
$ \frac{1}{n} \Tr (A_n^2)$ and $\frac{1}{n} \* \sum_{i=1}^n (A_n)_{ii}^2$ converge in probability to non-random limits
(see e.g. \cite{BGM} and the Appendix by J.Baik and J.Silverstein in \cite{CDF}). 
We note that $\|G\|\leq \frac{1}{|\Im z|}.$ 
The desired convergence in probability will easily follow from the 
self-averaging property of the resolvent entries established in Section \ref{sec:expvar} (see Proposition \ref{proposition:prop1}).
The generalization of the CLT to random $A_n$ is not unexpected since the distribution of $ \frac{1}{\sqrt{n}} \* (b^t \*A_n\*b - \Tr A_n)$ conditioned
on the matrix entries of $A_n$ is approximately $N(0, \kappa_4\*\frac{1}{n} \* \sum_{i=1}^n (A_n)_{ii}^2 + \frac{1}{n} \* \Tr (A_n^2))$ for large $n,$
and the expression in the variance converges to a non-random limit as $n\to \infty.$

As a result, one obtains that the term $\frac{1}{\sqrt{N}} \*\left(b^t \*G\*b - \sigma^2 \*\E (\Tr G)\right)$ in (\ref{ireland4}) 
converges in 
distribution as $N\to \infty$ to a complex Gaussian random variable with zero mathematical expectation and the covariance matrix given by the r.h.s. of
(\ref{dispersii1}-\ref{dispersii3}) with $w=z$ and $ \varphi_{++}(z,w), \ \varphi_{--}(z,w),$ and $\varphi_{+-}(z,w)$ defined in 
(\ref{padova1}-\ref{padova4}) in Section \ref{sec:formulation}. Since it is independent from $W_{NN},$  the limiting distribution of 
$\sqrt{N}\*(R_{NN}-g_{\sigma}(z))$ is given $g_{\sigma}(z)^2$ multiplied by the convolution of the marginal diagonal distribution $\mu_1$ and the 
complex Gaussian.

To study the joint distribution of several resolvent entries $R_{i_l j_l}(z_l), \ 1 \leq l \leq k, $ one follows a similar route. The main new ingredient 
is a multi-dimensional CLT for random bilinear (sesquilinear in the complex case)  forms (see Theorem 6.4. in \cite{BGM} 
and Theorem 7.3 in \cite{BY}; for the convenience of the 
reader  we reproduce the last one as Theorem \ref{thm:bggm} in Section \ref{section:appendix}).  
Thus, one is able to prove the result of Theorem \ref{thm:real} in Section \ref{sec:formulation} 
for the test functions of the form
\begin{equation}
\label{ffff}
f(x)=\sum_{i=1}^k c_k \frac{1}{z_k-x}, \ c_k\in \R, \ z_k \in \C, \ \Im z_k\not=0, \ 1\leq i\leq k.
\end{equation}

The finite fourth moment condition on the off-diagonal entries of $W_N$ and the finite second moment condition for the diagonal entries of $W_N$ imply 
that $\|X_N\|\to 2\*\sigma$ a.s. as $N\to \infty$ (\cite{BYin}).  Therefore, the limiting fluctuations of normalized matrix entries of $f(X_N)$ do not 
change if we alter $f$ outside $[-2\*\sigma-\delta, 2\*\sigma +\delta],$ where $\delta>0$ is an arbitrary fixed positive number.  In particular, one 
can replace $f$ by $f\*h$ where $h\in C^{\infty}(\R)$ is a function with compact support such that $h=1$ on $[-M, M]$ for sufficiently 
large $M.$  

In order to extend the result of Theorem \ref{thm:real} from (\ref{ffff}) to more
general test functions $f$, one approximates $f$ by test functions $f_m$ of the form
(\ref{ffff}) so that $\|f-f_m\|\to 0$ as $m\to \infty $  in an appropriate norm and 
$N\*\V(f(X_N)_{ij}-f_m(X_N)_{ij})$ goes to zero uniformly in $N$ when $m\to \infty,$  where $\V$ denotes the variance.

This program is the easiest to implement when the marginal 
distributions of the entries of $W_N$ satisfy the Poincar\'e inequality.  Indeed, for a Lipschitz test function $f,$ a matrix entry 
$f(X)_{ij}$ is a Lipschitz function of the matrix entries of $X$ (see e.g. \cite{CQT}).  Therefore, as a direct consequence of the Poincar\'e 
inequality for the marginal 
distributions of
$W_N$ one gets the bound
\begin{equation}
\label{liplip}
\V (f(X_N)_{ij})\leq \frac{|f|^2_{\mathcal{L}}}{\upsilon\*N},
\end{equation}
where $|f|_{\mathcal{L}}:=sup_{x\not=y} \frac{|f(x)-f(y)|}{|x-y|}$ is the Lipschitz constant.  We note that $N\*\V (f(X_N)_{ij})$ goes to zero if
$|f|_{\mathcal{L}}$ goes to zero.  Approximating a Lipschitz continuous compactly supported test function  $f$ by functions $f_m\* h$, where 
$f_m, \ m\geq 1, $ are of the form (\ref{ffff})
one finishes the proof.

If the marginal distributions do not satisfy the Poincar\'e inequality, one needs to impose some additional smoothness condition on $f$ to obtain an 
analogue of (\ref{liplip}).  We refer the reader to
the bound (\ref{chto3}) in Proposition \ref{proposition:prop2} in Section \ref{sec:formulation}. The proof of
(\ref{chto3}) consists of two steps.  First, one estimates
the variance of the resolvent entries and proves

\begin{equation*}
\V (R_{ij}(z)) = O \left( \frac{P_6(|\Im z|^{-1})}{N}\right), \  1\leq i, j\leq N, \ \text{uniformly on} \ \C\setminus \R,
\end{equation*}
where $P_6$ is some polynomial of degree $6$ with fixed positive coefficients (see (\ref{odinnadtsat102}) in Proposition \ref{proposition:prop1} in
Section \ref{sec:expvar}).  The proof of this bound is a bit long but quite standard and relies on the resolvent technique 
(see e.g. \cite{B}, \cite{CDF}, \cite{KKP}, \cite{LP1}, and \cite{Sh}).  In particular, many computations are similar to those used in 
the derivation of the master loop equation in the proof of the Wigner semicircle law by the resolvent method.

To extend the last estimate to more general test functions we use 
the Helffer-Sj\"{o}strand functional calculus discussed in Section \ref{sec:expvarreg} 
(see e.g. \cite{HS}, \cite{D}, \cite{HT}, or the proof of Lemma 5.5.5 in \cite{AGZ}). 
One then requires that $f$ have
four derivatives to compensate for the $|\Im z|^{-6}$ factor in the upper bound on $\V R_{ij}(z).$

The rest of the paper is organized as follows.  Section  \ref{sec:formulation} is devoted to formulation of our main results.  Perhaps, during the first 
glance at the paper the reader could just look at Theorem \ref{thm:real} in the real symmetric case (the analogue in the Hermitian case is Theorem
\ref{thm:herm}) and omit the rest of the section. 
Theorems \ref{thm:resreal} and \ref{thm:resherm} deal with the resolvent entries and are important building blocks in the proofs of
Theorems \ref{thm:real} and \ref{thm:herm}.  Theorems \ref{thm:resreal1}, \ref{thm:real1}, \ref{thm:resherm1}, and \ref{thm:resreal1} prove 
slightly stronger results since as they assume that the marginal distributions satisfy the Poincar\'e inequality.

The actual proof starts in Section \ref{sec:expvar} which is devoted to estimates on the mathematical expectation and the 
variance of the resolvent entries.  
The main 
results of this section are collected in Proposition \ref{proposition:prop1}. During the first reading 
of the paper, the reader might wish to skip long but rather straightforward computations in Section \ref{sec:expvar} and jump to the next 
section once the statement of Proposition \ref{proposition:prop1} is absorbed.

In Section \ref{sec:expvarreg}, we extend our 
estimates to the matrix entries $f(X_N)_{ij}$ for  sufficiently nice test functions $f$ by applying the Helffer-Sj\"{o}strand functional calculus. 
The main result of Section \ref{sec:expvarreg} is the proof of Proposition \ref{proposition:prop2}.

Section \ref{sec:flucmat} is devoted to studying the fluctuation of the resolvent entries and contains the 
proofs of Theorems \ref{thm:resreal}, \ref{thm:resreal1}, \ref{thm:resherm}, and \ref{thm:resherm1}.  The proofs of  Theorems \ref{thm:resreal} and
\ref{thm:resherm1} follow the route explained above when we discussed the fluctuation of a diagonal resolvent entry.  
The proofs of Theorems \ref{thm:resreal1} and 
\ref{thm:resherm1} could be omitted at the first reading as they prove the functional convergence in
a special case when the marginal distributions satisfy the Poincar\'e 
inequality.

Theorems \ref{thm:real} and \ref{thm:real1} are proved in Section \ref{sec:lipschitz}. The proofs of the corresponding 
results in the Hermitian  case, namely
Theorems \ref{thm:herm} and \ref{thm:herm1} is very similar and mostly left to the reader.

In the Appendix, we discuss various tools used throughout the paper.

We will denote throughout the paper by $const_i, Const_i, $ various positive constants
that may change from line to line.  Occasionally, we will drop the dependence on $N$ in the notations for the matrix entries.

We would like to thank M. Shcherbina for useful discussions and S. O'Rourke for careful reading of the manuscript and useful remarks.

\section{ \bf{Formulation of Main Results}}
\label{sec:formulation}

First, we consider the resolvent entries. In Theorems \ref{thm:resreal} and \ref{thm:resherm} formulated below, 
we study the limiting joint distribution of a 
finite number 
of resolvent matrix entries of a real symmetric (Hermitian) Wigner matrix $X_N$.  We recall that the resolvent $R_N(z)$ of $X_N$ is defined as
$$ R_N(z):=(z\*I_N-X_N)^{-1}, $$
for $z$ outside the spectrum of $X_N.$  We will be interested in the limiting joint distribution of a finite number of the resolvent entries.
Since the entries of $X_N$ are i.i.d. random variables up from the diagonal, we can study, without a loss of generality,
the joint distribution of the resolvent entries in an $m\times m$ upper-left corner 
of the matrix provided that $m$ is an arbitrary fixed positive integer.  

Let us denote by $R^{(m)}(z) $ the $m\times m$ upper-left corner of the matrix $R_N(z).$
In a similar fashion, we denote by $W^{(m)}, \ X^{(m)}, \ $ the $m\times m$ upper-left corner of matrices $W_N $ and $X_N,$ respectively.
If the reader is put off by some cumbersome formulas/notations in this section, he/she can always assume $m=1$ and deal with just one diagonal resolvent 
entry.  The case $m>1$ does not require any significant new ideas. 

Consider a matrix-valued random field 
\begin{equation}
\label{ups}
\Upsilon_N(z)=\sqrt{N} \* \left(R^{(m)}(z)-g_{\sigma}(z)\*I_m \right), \ z \in \C \setminus [-2\*\sigma, 2\*\sigma].
\end{equation}
Clearly, $\Upsilon_N(z)$ is a random function on  $ \C \setminus [-2\*\sigma, 2\*\sigma] $ with values in the space of
complex symmetric $m\times m$ matrices ($\Upsilon_N(x)$ is real symmetric for real $ x$).
Let us define
\begin{equation}
\label{padova1}
\varphi(z,w):= \int_{-2\sigma}^{2\sigma} \frac{1}{z-x}\* \frac{1}{w-x}  \* \frac{1}{2 \pi \sigma^2} \sqrt{ 4 \sigma^2 - x^2}  \* dx=
\left \{\begin{array} {r@{\quad:\quad}l} 
-\frac{g_{\sigma}(w)-g_{\sigma}(z)}{w-z} & \text{if} \ \ w\not=z, \\
-g_{\sigma}'(z)& \text{if} \ \ w=z. \end{array} \right.
\end{equation}
for $z,w \in \C \setminus [-2\*\sigma, 2\*\sigma].\ $  Clearly, $\varphi(z,w)=\E\left( \frac{1}{z-\eta}\*\frac{1}{w-\eta} \right), \ $
where $\eta $ is distributed according to the Wigner semicircle law (\ref{polukrug}).
We also define
\begin{align}
\label{padova2}
& \varphi_{++}(z,w):= \int_{-2\sigma}^{2\sigma} \Re \frac{1}{z-x}\* \Re\frac{1}{w-x}  
\* \frac{1}{2 \pi \sigma^2} \sqrt{ 4 \sigma^2 - x^2} \* dx\\
& =\frac{1}{4}\* \left(\varphi(z,w)+\varphi(\bar{z},\bar{w})+\varphi(\bar{z},w)+\varphi(z,\bar{w})\right), \nonumber\\
\label{padova3}
& \varphi_{--}(z,w):= \int_{-2\sigma}^{2\sigma} \Im \frac{1}{z-x}\* \Im\frac{1}{w-x}  
\* \frac{1}{2 \pi \sigma^2} \sqrt{ 4 \sigma^2 - x^2}  \* dx\\
& = -\frac{1}{4}\* \left(\varphi(z,w)+\varphi(\bar{z},\bar{w})-\varphi(\bar{z},w)-\varphi(z,\bar{w})\right), \nonumber\\
\label{padova4}
& \varphi_{+-}(z,w):= \int_{-2\sigma}^{2\sigma} \Re \frac{1}{z-x}\* \Im\frac{1}{w-x}  
\* \frac{1}{2 \pi \sigma^2} \sqrt{ 4 \sigma^2 - x^2}  \* dx\\
& =-\frac{i}{4}\* \left(\varphi(z,w)+\varphi(\bar{z},w)-\varphi(\bar{z},\bar{w})-\varphi(z,\bar{w})\right), \nonumber
\end{align}

\begin{theorem}
\label{thm:resreal}
Let $X_N=\frac{1}{\sqrt{N}} W_N$ be a random real symmetric Wigner matrix 
(\ref{offdiagreal1}-\ref{diagreal2}).
The random field 
$ \Upsilon_N(z) $ in (\ref{ups}) converges in finite-dimensional distributions to a random field 
\begin{equation}
\label{functionalconv}
\Upsilon (z) = g_{\sigma}^2(z)\*(W^{(m)} + Y(z)),
\end{equation}
where $W^{(m)}$ is the $m\times m$ upper-left corner submatrix of a Wigner matrix $W_N, $ and
$Y(z)=\left(Y_{ij}(z)\right), Y_{ij}(z)=Y_{ji}(z), \ 1\leq i,j \leq m, $  is a Gaussian random field such that 

\begin{align}
\label{dispersii1}
& \Cov(\Re Y_{jj}(z), \Re Y_{jj}(w))= \kappa_4(\mu)\* \Re g_{\sigma}(z) \* \Re g_{\sigma}(w) +2\*\sigma^4\* \varphi_{++}(z,w), \\
\label{dispersii2}
& \Cov(\Im Y_{jj}(z), \Im Y_{jj}(w))= \kappa_4(\mu)\* \Im g_{\sigma}(z) \*\Im g_{\sigma}(w) +
2\*\sigma^4\* \varphi_{--}(z,w), \\
\label{dispersii3}
& \Cov(\Re Y_{jj}(z), \Im Y_{jj}(w) )= \kappa_4(\mu)\* \Re g_{\sigma}(z) \* \Im g_{\sigma}(w)  
+ 2 \*\sigma^4\* \varphi_{+-}(z,w), \\
\label{dispersii4}
& \Cov(\Re Y_{ij}(z),  \Re Y_{ij}(w))= \sigma^4\* \varphi_{++}(z,w), \ i\not=j,\\
\label{dispersii5}
& \Cov(\Im Y_{ij}(z), \Im Y_{ij}(w))= \sigma^4\* \varphi_{--}(z,w), \ i\not=j,\\
\label{dispersii6}
& \Cov(\Re Y_{ij}(z), \Im Y_{ij}(w) )= \sigma^4\* \varphi_{+-}(z,w), \ i\not=j,
\end{align}
where the fourth cumulant $\kappa_4(\mu):= \int x^4 \*\mu(dx) - 3\* (\int x^2 \* \mu(dx))^2=m_4-3\sigma^4.$  

In addition, for any finite $r\geq 1, $ the 
entries $Y_{i_lj_l}(z_1), \ 1\leq i_l\leq j_l \leq m, \ 1\leq l\leq r,$ are independent provided
$(i_{l_1}, j_{l_1})\not=(i_{l_2}, j_{l_2})$ for $1\leq l_1\not=l_2\leq r.$
\end{theorem}

\begin{remark}
\label{xreal}
If $z=x \in \R \setminus [-2\*\sigma, 2 \*\sigma], \ $ then $Y_{ij}(x), \ 1\leq i \leq j \leq m, \ $ are  independent centered real Gaussian random 
variables with the variance given by
\begin{align}
\label{Disp1} 
& \V(Y_{ii}(x))= \kappa_4(\mu)\* g_{\sigma}^2(x) - 2\*\sigma^4\* g_{\sigma}'(x), \ 1\leq i \leq m, \\
\label{Disp2}
& \V(Y_{ij}(x))= -\sigma^4\* g_{\sigma}'(x), \ 1\leq i <j \leq m.
\end{align}
\end{remark}

Let $ \mathcal{D} \subset \C \setminus [-2\*\sigma, 2\*\sigma] $ be a compact set. The distribution of $\Upsilon_N(z), z\in \mathcal{D}, $ 
defines a probability measure $ \mathcal{P_N} $ on $C(\mathcal{D},\C^{m(m+1)/2}).$ 
One can prove functional convergence in distribution 
for the random field $\Upsilon_N(z), \ z\in \mathcal{D}, $ provided $\mu $ and $\mu_1$ satisfy some additional conditions on the decay of their tail 
distributions.  For simplicity, we will consider the case when $\mu $ and $\mu_1$ satisfy the Poincar\'e inequality.

We recall that a probability measure  $\P$ on $\R^M$ satisfies the Poincar\'e inequality with constant $\upsilon>0$ if, for all continuously differentiable 
functions 
$f: \R^M \to \C,$
\begin{equation}
\label{poin}
\V_{\P}(f)= \E_{\P} \left(|f(x)-\E_{\P}(f(x))|^2\right) \leq \frac{1}{\upsilon}\* \E_{\P}[ |\nabla f(x)|^2 ]
\end{equation} 
Note that the Poincar\'e inequality tensorizes and the probability measures satisfying the Poincar\'e inequality
have sub-exponential tails (\cite{GZ}, \cite{AGZ}) . 
By a standard scaling argument, we note that if the marginal distributions $\mu $ and $\mu_1$ 
of the matrix entries of $W_N$  satisfy the Poincar\'e inequality
with constant
$\upsilon>0$ then the marginal distributions of the matrix entries of $X_N=\frac{1}{\sqrt{N}}\*W_N $ satisfy the Poincar\'e inequality with constant
$N\*\upsilon.\ $

\begin{theorem}
\label{thm:resreal1}
Let $X_N=\frac{1}{\sqrt{N}} W_N$ be a random real symmetric Wigner matrix 
(\ref{offdiagreal1}-\ref{diagreal2}) and the marginal distributions $\mu $ and $\mu_1$ satisfy the
Poincar\'e inequality (\ref{poin}).
The probability measure $ \mathcal{P_N} $ on $C(\mathcal{D},\C^{m(m+1)/2}) $ given by the random field 
$ \Upsilon_N(z) $ in (\ref{ups}) weakly converges to the distribution of the random field $\Upsilon (z)$ defined in Theorem \ref{thm:resreal}.
\end{theorem}

Next, we extend the results of Theorem \ref{thm:resreal} to the matrix entries of $f(X_N) \ $  
for regular $f.\ $ 
We say that a function $f:\R \to  \R$ belongs to $C^n(I),$ if $f$ and its first $n$ derivatives are continuous on 
$I.$  We will use the notation
$C^n_c(\R)$ for the space of $n$ times continuously differentiable functions on $\R$ with compact support.
We define the norm on $C^n(I)$ for compact $I\subset \R$ as
\begin{equation}
\label{normaa}
\|f\|_{C^n(I)}:=\max( |d^kf/dx^k(x)|, \ 0\leq k \leq n, \ x\in I).
\end{equation}
We also define for $f \in C^n(\R)$
\begin{equation}
\label{normsob}
\|f\|_{n,1}:= \max( \int_{-\infty}^{\infty} |d^kf/dx^k(x)| \* dx, \ 0\leq k \leq n),
\end{equation}
and
\begin{equation}
\label{normasobolev}
\|f\|_{n,1,+}:=\max \left( \int_{\R} (|x|+1) \*|\frac{d^lf}{dx^l}(x)| \* dx, \ 0\leq l \leq n \right).
\end{equation}
Clearly, the right hand sides of (\ref{normsob}) and (\ref{normasobolev}) could be infinite.

We start with Proposition \ref{proposition:prop2} that holds both in the real symmetric and Hermitian cases.
\begin{proposition}
\label{proposition:prop2}
Let $X_N=\frac{1}{\sqrt{N}} W_N$ be a random real symmetric (Hermitian) Wigner matrix defined in (\ref{offdiagreal1}-\ref{diagreal2})
((\ref{offdiagherm1}-\ref{diagherm1})).

(i) Let $L$ be some positive number and $f\in C^7_c(\R)$ with $supp(f) \subset [-L, +L].$
Then there exists a constant $Const(L, \mu, \mu_1)$ that depends on $L, \mu, $ and $\mu_1,$ such that
\begin{equation}
\label{chto1}
|\E (f(X_N)_{ii})-\int_{-2\sigma}^{2\sigma} f(x) \* \frac{1}{2 \pi \sigma^2} \sqrt{ 4 \sigma^2 - x^2}  \* dx|
\leq Const(L,\mu, \mu_1)\* \frac{\|f\|_{C^7([-L, L])}}{N}, \ 1\leq i \leq N.
\end{equation}

(ii) Let $f \in C^8(\R), $
then
\begin{align}
\label{chto11}
& |\E (f(X_N)_{ii})-\int_{-2\sigma}^{2\sigma} f(x) \* \frac{1}{2 \pi \sigma^2} \sqrt{ 4 \sigma^2 - x^2}  \* dx|  \\
& \leq Const(\mu, \mu_1) \* \frac{\|f\|_{8,1,+}}{N} ,
\ 1\leq i \leq N. \nonumber
\end{align}

(iii) Let $f\in C^6(\R),$ then there exists a constant $Const(\mu, \mu_1)$ such that
\begin{equation}
\label{chto2}
|\E (f(X_N)_{jk})|\leq Const(\mu, \mu_1) \* \frac{\|f\|_{6,1}}{N}, \ 1\leq j <k\leq N.
\end{equation}

(iv) Let $f\in C^4(\R),$ then there exists a constant $Const(\mu, \mu_1)$ such that
\begin{equation}
\label{chto3}
\V (f(X_N)_{ij})\leq Const(\mu, \mu_1) \*\frac{\|f\|^2_{4,1}}{N}, \ 1\leq i,j \leq N.
\end{equation}

(v) If $\mu$ has finite fifth moment, $\mu_1$ has finite third moment, 
and $f \in C^{10}(\R), $
then one can improve (\ref{chto2}), namely
\begin{equation}
\label{chto4}
|\E (f(X_N)_{jk})|\leq Const(\mu, \mu_1) \* \frac{\|f\|_{10,1}}{N^{3/2}}, \ 1\leq j <k\leq N.
\end{equation}
\end{proposition}

\begin{remark}
In \cite{PRS}, we extend the results of Propositions \ref{proposition:prop2} to the case of 
$\langle u^{(N)}, f(X_N) \* v^{(N)} \rangle, \ $  where $u^{(N)} $ and $v^{(N)} $ are arbitrary nonrandom  vectors from $\C^N.$
\end{remark}

In order to formulate our next theorem, we need to introduce several notations.
Recall that we defined $\omega^2(f) $ in (\ref{omegasq}) as
$$ \omega^2(f)= \V(f(\eta))= \frac{1}{2}\* \int_{-2\sigma}^{2\sigma} \int_{-2\sigma}^{2\sigma} (f(x)-f(y))^2 
\*  \frac{1}{4 \pi^2 \sigma^4} \sqrt{ 4 \sigma^2 - x^2} \* 
\sqrt{ 4 \sigma^2 - y^2} \* dx \* dy, $$ where $\eta$ is distributed according to the Wigner semicircle law (\ref{polukrug}).
In addition, we define
\begin{equation}
\label{alphaf}
\alpha(f):=\E\left(f(\eta)\* \frac{\eta}{\sigma}\right)= 
\frac{1}{\sigma}\* \int_{-2\sigma}^{2\sigma} x\*f(x)\* \frac{1}{2 \pi \sigma^2} \sqrt{ 4 \sigma^2 - x^2}  \* dx
\end{equation}
and
\begin{equation}
\label{betaf}
\beta(f):= \E\left(f(\eta)\*\frac{\eta^2-\sigma^2}{\sigma^2}\right)= \frac{1}{\sigma^2}\int_{-2\sigma}^{2\sigma} f(x)\*(x^2-\sigma^2)\* 
\frac{1}{2 \pi \sigma^2} \sqrt{ 4 \sigma^2 - x^2} .
\end{equation}

We recall that a function $f:\R \to  \R$ is called Lipschitz continuous on an interval $I \subset \R $ if there exists a constant $L $ such that
\begin{equation}
\label{funkciyaLipschitz}
|f(x)-f(y)| \leq L\*|x-y|, \ \ \text{ for all} \ x,y \in I.
\end{equation}
We define
\begin{equation}
\label{snova}
|f|_{\mathcal{L}, \R}=\sup_{x\not=y}\*\frac{|f(x)-f(y)|}{|x-y|},
\end{equation}
and
\begin{equation}
\label{snova1}
|f|_{\mathcal{L}, \delta}=\sup_{x\not=y, \ x,y \in [-2\sigma -\delta, 2\sigma +\delta]}\*\frac{|f(x)-f(y)|}{|x-y|}.
\end{equation}

Finally, let us introduce a $C^{\infty}(\R) $ function $h(x)$ with compact support such that
\begin{equation}
\label{amerika}
h(x)\equiv 1 \ \text{for} \  x \in [-2\*\sigma -\delta, 2\*\sigma +\delta], \ \delta>0.
\end{equation}

\begin{theorem}
\label{thm:real}
Let $X_N=\frac{1}{\sqrt{N}} W_N$ be a random real symmetric Wigner matrix
(\ref{offdiagreal1}-\ref{diagreal2}).
Let $f:\R \to  \R $ be four times continuously differentiable on $[-2\*\sigma -\delta, 2\*\sigma +\delta] $ for some $\delta>0.\ $  
Then the following holds.

(i) For $ i =j, $
\begin{equation}
\label{trudno1}
\sqrt{N} \*\left( f(X_N)_{ii} - \E\left( (fh)(X_N)_{ii}\right)   \right) 
\end{equation}
converges in distribution to the sum of two independent  random variables \\
$ \frac{\alpha(f)}{\sigma}\* W_{ii} $ and $N(0, v_1^2(f)), $
where $h$ is an arbitrary $C^{\infty}_c(\R) $ function satisfying (\ref{amerika}), and
\begin{equation}
\label{vsqf}
v_1^2(f):= 2 \* \left(\omega^2(f) -\alpha^2(f) +\frac{\kappa_4(\mu)}{2\*\sigma^4} \* \beta^2(f)\right).
\end{equation}
If $f$ is seven times continuously differentiable on $[-2\sigma-\delta, 2\sigma+\delta], $ then the statement still holds if one replaces
$\E\left( (fh)(X_N)_{ii}\right)$ in (\ref{trudno1}) by
\begin{equation}
\label{c1f}
\int_{-2\sigma}^{2\sigma} f(x) \* \frac{1}{2 \pi \sigma^2} \sqrt{ 4 \sigma^2 - x^2}  \* dx.
\end{equation}

(ii) For $ i \not=j, $
\begin{equation}
\label{trudno2}
\sqrt{N} \* \left( f(X_N)_{ij} - \E\left( (fh)(X_N)_{ij}\right)\right) 
\end{equation}
converges in distribution to the sum of two independent  random variables \\
$ \frac{\alpha(f)}{\sigma} \*W_{ij} $ and $N(0, d^2(f)), $
with
\begin{equation}
\label{dsqf}
d^2(f):= \omega^2(f) -\alpha^2(f).
\end{equation}
If $f$ is six times continuously differentiable on $[-2\sigma-\delta, 2\sigma+\delta], $ then one can replace
$\E\left( (fh)(X_N)_{ij}\right)$ in (\ref{trudno2}) by $0.$

(iii) For any finite $m, $ the normalized matrix entries 
\begin{equation}
\label{mnogo}
\sqrt{N} \*\left( f(X_N)_{ij} - \E ((fh)(X_N)_{ij})\right), \ 1\leq i \leq j \leq m, 
\end{equation}
are independent in the limit $N \to \infty. $ 
\end{theorem}
We follow Theorem \ref{thm:real} with several remarks.

\begin{remark}
If $f \in C^4(\R), $ and $\|f\|_{4,1} <\infty, $ where $\|f\|_{4,1}$ is defined in (\ref{normsob}), then
it follows from Proposition \ref{proposition:prop2} that
the centralizing constants in (\ref{trudno1}) and (\ref{trudno2})
can be taken to be $\E (f(X_N))_{ij}. $
\end{remark}

\begin{remark}
\label{ZAM}
It follows from the definition of the fourth cumulant that $\frac{\kappa_4(\mu)}{2\*\sigma^4} \geq -1, $ with the equality taking place only when
$\mu$ is Bernoulli.  Since $ 1, \ \frac{x}{\sigma}, \ $ and $ \frac{x^2-\sigma^2}{\sigma^2}, $ are the first three orthonormal polynomials associated with
the semicircle measure (\ref{polukrug}), it follows from the Bessel inequality 
that for the diagonal entries the variance $v_1^2(f)$ of the Gaussian component in (\ref{vsqf}) is zero if and only if $f(x) $ 
is either a linear function of $x$ or $f(x)$ a quadratic polynomial and $\mu $ Bernoulli.  Similarly, for the off-diagonal entries one has that
the variance $d^2(f)$ of the Gaussian component in (\ref{dsqf}) is zero if and only if $f(x)$ is linear.

The statement of the Theorem \ref{thm:real} also implies that the limiting distribution of 
the normalized $(ij)$th entry of $f(X_N)$ is Gaussian
if and only if either $\mu $ for $i\not=j$ (correspondingly $\mu_1$ for $i=j$) is Gaussian or $\E(\eta\*f(\eta))=0.\ $ The same holds 
in the Hermitian case.
\end{remark}

\begin{remark}
Utilizing Proposition 1 in \cite{Sh}, one can extend the result of Theorem \ref{thm:real} to the test functions satisfying
\begin{equation}
\label{sobolev}
\|f\|_s^2 = \int (1+|k|)^{2s} \* |\hat{f}(k)|^2 \* dk < \infty, \ s>3, \ \ \hat{f}(k):=\int e^{-2\*\pi\*k\*x}\* f(x) \*dx,
\end{equation}
and, more generally, to the functions that coincide with the functions in (\ref{sobolev}) on $[-2\*\sigma-\delta, 2\*\sigma +\delta]$
for some $\delta>0 $ (\cite{ORS}).
\end{remark}

\begin{remark}
There has been a significant body of work on the Central Limit Theorem for $\Tr f(X_N)= \sum_{i=1}^N f(X_N)_{ii}. \ $  We refer the reader to 
\cite{AZ}, \cite{BY}, \cite{BWZ}, \cite{LP1}, and \cite{Sh}, and the references therein.  In particular, in \cite{Sh}, 
the CLT is proved assuming that the fourth moment of the marginal distribution is finite and
$\|f\|_s <\infty$ for $s>3/2.$
\end{remark}

If $\mu $ and $\mu_1$ satisfy the Poincar\'e inequality, one can prove convergence in distribution for the matrix entries of Lipschitz continuous test 
functions.

\begin{theorem}
\label{thm:real1}
Let $X_N=\frac{1}{\sqrt{N}} W_N$ be a random real symmetric Wigner matrix 
(\ref{offdiagreal1}-\ref{diagreal2}) and
the marginal distributions $\mu $ and $\mu_1$ satisfy the
Poincar\'e inequality (\ref{poin}). Then the following holds.

(i) If $f:\R\to\R$ is Lipschitz continuous on $[-2\*\sigma -\delta, 2\*\sigma +\delta] $ that
satisfies a sub-exponential growth condition
\begin{equation}
\label{exprost}
|f(x)|\leq a \* \exp(b\*|x|) \ \ \text{ for all \ } x \in \R
\end{equation}
for some positive constants $a$ and $b,$
then the results of Theorem \ref{thm:real} hold with the centralizing constants $\E (f(X_N)_{ij}))$
in (\ref{trudno1}) and (\ref{trudno2}).

Moreover,
\begin{align}
\label{sanjosesh1}
&  \P \left( |f(X_N)_{ij}- \E(f(X_N)_{ij})| \geq t \right) \\
& \leq  2\*K \* \exp\left( -\frac{\sqrt{\upsilon \*N} \* t}{2 \*|f|_{\mathcal{L},\delta}} \right) + (2\*K+ o(1))\* 
\exp\left(-\frac{\sqrt{\upsilon\*N}}{2}\*\delta\right), \nonumber
\end{align}
where $|f|_{\mathcal{L},\delta} $ is defined in (\ref{snova1}),
\begin{equation}
\label{KK}
K=-\sum_{i\geq 0} 2^i\*\log(1-2^{-1}\*4^{-i}),
\end{equation}
and $\upsilon $ is the constant in the Poincar\'e inequality (\ref{poin}).

(ii) If $f\in C^7(\R)$ (correspondingly, $f\in C^6(\R), \ f\in C^{10}(\R)) $ and $f$ satisfies the subexponential growth condition
(\ref{exprost}), then the estimate (\ref{chto1}) (correspondingly, (\ref{chto2}),(\ref{chto4})) from Proposition \ref{proposition:prop2} holds.

(iii) If $f$ is a Lipschitz continuous function on $\R,$
then \begin{align}
\label{sanjosesh2}
&  \P \left( |f(X_N)_{ij}- \E(f(X_N)_{ij})| \geq t \right) \\
& \leq  2\*K \* \exp\left( -\frac{\sqrt{\upsilon \*N} \* t}{2\*|f|_{\mathcal{L}, \R}} \right), \nonumber
\end{align}
where $|f|_{\mathcal{L},\R} $ is defined in (\ref{snova}).
\end{theorem}

\begin{remark}
If $f$ is a Lipschitz continuous function on $[-2\*\sigma -\delta, 2\*\sigma +\delta] $ that
does not satisfy the subexponential growth condition (\ref{exprost}) then the results of 
Theorem \ref{thm:real} still hold with the centralizing constants $\E \left((fh)(X_N)_{ij})\right))$
in (\ref{trudno1}) and (\ref{trudno2}).
\end{remark}

In the second part of this section, we formulate the analogous results in the 
the Hermitian case.  

As in the real symmetric case, consider a matrix-valued random field 
$$\Upsilon_N(z)=\sqrt{N} \* \left(R^{(m)}(z)-g_{\sigma}(z)\*I_m \right), \ z \in \C \setminus [-2\*\sigma, 2\*\sigma].$$
Clearly, $\Upsilon_N(z)$ is a random function on  $ \C \setminus [-2\*\sigma, 2\*\sigma] $ with values in the space of
complex $m\times m$ matrices. $\Upsilon_N(x)$ is Hermitian for real $ x$ and, more generally, $\Upsilon_N(z)=\Upsilon_N(\bar{z})^*. $

\begin{theorem}
\label{thm:resherm}
Let $X_N=\frac{1}{\sqrt{N}} W_N$ be a Hermitian Wigner matrix (\ref{offdiagherm1}-\ref{diagherm1}).
The random field 
$ \Upsilon_N(z) $ converges in finite-dimensional distributions to a random field 
\begin{equation}
\label{functionalconv1}
\Upsilon (z) = g_{\sigma}^2(z)\*(W^{(m)} + Y(z)),
\end{equation}
where $W^{(m)}$ is the $m\times m$ upper-left corner submatrix of a Wigner matrix $W_N, $ and
$Y(z)=\left(Y_{ij}(z)\right), \ 1\leq i,j \leq m, \ $  is a Gaussian random field such that 

\begin{align}
\label{dispersii11}
& \Cov(\Re Y_{jj}(z), \Re Y_{jj}(w))= \kappa_4(\mu)\* \Re g_{\sigma}(z) \* \Re g_{\sigma}(w) +\sigma^4\* \varphi_{++}(z,w), \\
\label{dispersii12}
& \Cov(\Im Y_{jj}(z), \Im Y_{jj}(w))= \kappa_4(\mu)\* \Im g_{\sigma}(z) \*\Im g_{\sigma}(w) +
\sigma^4\* \varphi_{--}(z,w), \\
\label{dispersii13}
& \Cov(\Re Y_{jj}(z), \Im Y_{jj}(w) )= \kappa_4(\mu)\* \Re g_{\sigma}(z) \* \Im g_{\sigma}(w)  
+ \sigma^4\* \varphi_{+-}(z,w), \\
\label{dispersii14}
& \Cov(\Re Y_{ij}(z),  \Re Y_{ij}(w))= \frac{1}{2}\*\sigma^4\* (\varphi_{++}(z,w)+\varphi_{--}(z,w)), \ i\not=j, \\
\label{dispersii15}
& \Cov(\Im Y_{ij}(z), \Im Y_{ij}(w))= \frac{1}{2}\*\sigma^4\* (\varphi_{++}(z,w)+\varphi_{--}(z,w)), \ i\not=j,\\
\label{dispersii16}
& \Cov(\Re Y_{ij}(z), \Im Y_{ij}(w) )= \frac{1}{2}\*\sigma^4\* (\varphi_{+-}(z,w)-\varphi_{+-}(w,z)), \ i\not=j.
\end{align}
where the fourth cumulant $\kappa_4(\mu):= \E |(W_N)_{12}|^4- 2\*\E|(W_N)_{12}|^2=m_4-2\sigma^4.$

In addition, for any finite $r\geq 1,$ the 
entries $Y_{i_lj_l}(z_1), \ 1\leq i_l\leq j_l \leq m, \ 1\leq l\leq r,$ are independent provided 
$(i_{l_1}, j_{l_1})\not=(i_{l_2}, j_{l_2}) $ for $1\leq l_1\not=l_2\leq r.$
\end{theorem}

\begin{remark}
\label{xherm}
If $z=x \in \R \setminus [-2\*\sigma, 2 \*\sigma], \ $ then 
$ Y_{ll}(x), \ 1\leq l \leq m, \ \Re Y_{ij}(x), \ \Im Y_{ij}(x),\\
1\leq i < j \leq m, \ $ are  independent centered real Gaussian random 
variables with the variance given by
\begin{align}
\label{Disp11} 
& \V(Y_{ll}(x))= \kappa_4(\mu)\* g_{\sigma}^2(x) - \sigma^4\* g_{\sigma}'(x), \ 1\leq l \leq m, \\
\label{Disp12}
& \V(\Re Y_{ij}(x))= -\frac{1}{2}\*\sigma^4\* g_{\sigma}'(x), \ \ \V(\Im Y_{ij}(x))= -\frac{1}{2}\*\sigma^4\* g_{\sigma}'(x), \ 1\leq i <j \leq m, \\
\label{Disp13}
& \Cov(\Re Y_{ij}(x), \Im Y_{ij}(x))=0, \ 1\leq i <j \leq m.
\end{align}
\end{remark}

Let $ \mathcal{D}, $ as before, be a compact subset of $\C \setminus [-2\*\sigma, 2\*\sigma]. $
The distribution of $\Upsilon_N(z), z\in \mathcal{D}, $ 
defines a probability measure $ \mathcal{P_N} $ on $C(\mathcal{D},\C^{m(m+1)/2}). $ 

\begin{theorem}
\label{thm:resherm1}
Let $X_N=\frac{1}{\sqrt{N}} W_N$ be a random Hermitian Wigner matrix
(\ref{offdiagherm1}-\ref{diagherm1}) and the marginal distributions $\mu $ and $\mu_1$ satisfy the
Poincar\'e inequality (\ref{poin}).
The probability measure $ \mathcal{P_N} $ on $C(\mathcal{D},\C^{m(m+1)/2}) $ given by the random field 
$ \Upsilon_N(z) $ weakly converges to the distribution of the random field $\Upsilon (z)$ defined in Theorem \ref{thm:resherm}.
\end{theorem}

Next theorem extends the results of Theorems \ref{thm:real} and \ref{thm:real1} to the Hermitian case.  We recall that 
$ \omega^2(f), \alpha(f), \beta(f),$ and $ d^2(f)$ have been defined in (\ref{omegasq}), (\ref{alphaf}), (\ref{betaf}), and (\ref{dsqf}).
\begin{theorem}
\label{thm:herm}
Let $X_N=\frac{1}{\sqrt{N}} W_N$ be a random Hermitian Wigner matrix 
(\ref{offdiagherm1}-\ref{diagherm1}).
Let $f:\R \to  \R $ be four times continuously differentiable on $[-2\*\sigma -\delta, 2\*\sigma +\delta] $ for some $\delta>0.\ $  
Then the following holds.

(i) For $ i =j, $
\begin{equation}
\label{trudno11}
\sqrt{N} \*\left( f(X_N)_{ii} - \E\left( (fh)(X_N)_{ii}\right)   \right) 
\end{equation}
converges in distribution to the sum of two independent  random variables \\
$ \frac{\alpha(f)}{\sigma}\* W_{ii} $ and $N(0, v_2^2(f)), $
where $h$ is an arbitrary $ C^{\infty}_c(\R) $ function satisfying (\ref{amerika}),
\begin{equation}
\label{vsqfherm}
v_2^2(f):= \omega^2(f) -\alpha^2(f) +\frac{\kappa_4(\mu)}{\sigma^4} \* \beta^2(f),
\end{equation}
and
\begin{equation*}
\kappa_4(\mu)= \E |(W_N)_{12}|^4- 2\*\E|(W_N)_{12}|^2=m_4-2\sigma^4.
\end{equation*}

If $f$ is seven times continuously differentiable on $[-2\sigma-\delta, 2\sigma+\delta], $ then one replace $\E\left( (fh)(X_N)_{ii}\right) $
in (\ref{trudno11}) by (\ref{c1f}).

(ii) For $ i \not=j, $
\begin{equation}
\label{trudno12}
\sqrt{N} \* \left( f(X_N)_{ij} - \E\left( (fh)(X_N)_{ij}\right)     \right) 
\end{equation}
converges in distribution to the sum of two independent  random variables \\
$ \frac{\alpha(f)}{\sigma} \*W_{ij} \ $ and complex Gaussian $N(0, d^2(f))$
with i.i.d real and imaginary parts $N(0, \frac{1}{2}\*d^2(f)).$
If $f$ is six times continuously differentiable on $[-2\sigma-\delta, 2\sigma+\delta], $ then one can replace $\E\left( (fh)(X_N)_{ij}\right)$
in (\ref{trudno12}) by $0.$

(iii) For any finite $m, $ the normalized matrix entries 
\begin{equation}
\label{mnogoh}
\sqrt{N} \*\left( f(X_N)_{ij} - \E ((fh)(X_N)_{ij})\right), \ 1\leq i \leq j \leq m, 
\end{equation}
are independent in the limit $N \to \infty. $ 
\end{theorem}

\begin{theorem}
\label{thm:herm1}  
Let $X_N=\frac{1}{\sqrt{N}} W_N$ be a random Hermitian Wigner matrix 
(\ref{offdiagherm1}-\ref{diagherm1})
and the marginal distributions $\mu $ and $\mu_1$ satisfy the
Poincar\'e inequality (\ref{poin}).  Then the following holds

(i) If $f$ is a Lipschitz continuous function on $[-2\*\sigma -\delta, 2\*\sigma +\delta] $ that
satisfies the subexponential growth condition (\ref{exprost}),
then the results of (\ref{thm:herm}) hold with the centralizing constants $\E\left( (fh)(X_N)_{ij}\right) $ in (\ref{trudno11}) and (\ref{trudno12}).

Moreover,
\begin{align}
\label{sanjosesh3}
&  \P \left( |f(X_N)_{ij}- \E(f(X_N)_{ij})| \geq t \right) \\
& \leq  2\*K \* \exp\left( -\frac{\sqrt{\upsilon \*N} \* t}{\sqrt{2}\*|f|_{\mathcal{L},\delta}} \right) + (2\*K+ o(1))\* 
\exp\left(-\frac{\sqrt{\upsilon\*N}}{\sqrt{2}}\*\delta\right), \nonumber
\end{align}
where $K$ is as in (\ref{KK}), $\upsilon$ is the constant in the Poincar\'e inequality (\ref{poin}), and $|f|_{\mathcal{L},\delta}$ defined 
in (\ref{snova1}).

(ii) If $f\in C^7(\R)$ (correspondingly, $f\in C^6(\R), \ f\in C^{10}(\R))$ and $f$ satisfies the subexponential growth condition
(\ref{exprost}), then the estimate (\ref{chto1}) (correspondingly, (\ref{chto2}),(\ref{chto4}))  from Proposition \ref{proposition:prop2} holds.

(iii) If the marginal distributions $\mu $ and $\mu_1$ satisfy the
Poincar\'e inequality (\ref{poin}) and $f$ is a Lipschitz continuous function on $\R,$
then
\begin{align}
\label{sanjoses4}
&  \P \left( |f(X_N)_{ij}- \E(f(X_N)_{ij})| \geq t \right) \\
& \leq  2\*K \* \exp\left( -\frac{\sqrt{\upsilon \*N} \* t}{\sqrt{2}\*|f|_{\mathcal{L},\R}} \right), \nonumber
\end{align}
where $|f|_{\mathcal{L},\R}$ is defined 
in (\ref{snova}).
\end{theorem}

Most of the proofs will be given in the real symmetric case.  The proofs in the Hermitian are essentially the same.
\begin{remark}
Theorems \ref{thm:resreal},   \ref{thm:real}, \ref{thm:resherm},   \ref{thm:herm}  
can be extended to the case when the matrix entries $ (W_N)_{ij}, \ 1\leq i\leq j\leq N, $ are independent  but not identically distributed \cite{ORS}.
In the real symmetric case, one requires that the off-diagonal entries satisfy
\begin{align*}
& \E (W_N)_{ij}=0, \ \V((W_N)_{ij})=\sigma^2, \ 1\leq i<j\leq N, \ \sup_{N, i\not=j} \E((W_N)_{ij})^4=m_4<\infty, \nonumber \\
& m_4(i):=\lim_{N\to \infty} \* \frac{1}{N}\* \sum_{j: j\not=i} \* \E |(W_N)_{ij}|^4 \ \text{exists for} \ 1\leq i \leq m, \nonumber\\
& L_N(\epsilon) \to 0, \ \text{as} \ N\to \infty, \ \forall \epsilon>0, \ \text{where} \\
& L_N(\epsilon)= \frac{1}{N^2}\* \sum_{1\leq i<j\leq N} \E\left( |(W_N)_{ij}|^4 \*\chi (|(W_N)_{ij}|\geq \epsilon\*\sqrt{N})\right ),\\
& L_{i,N}(\epsilon) \to 0, \ \text{as} \ N\to \infty, \ \forall \epsilon>0, \ 1\leq i \leq m, \ \text{where}\\
& L_{i,N}(\epsilon)= \frac{1}{N}\* \sum_{j:j\not=i} \E\left( |(W_N)_{ij}|^4 \*\chi(|(W_N)_{ij}|\geq \epsilon\*N^{1/4})\right ),
\end{align*}
and the diagonal entries satisfy
\begin{align*}
& \E (W_N)_{ii}=0, \ \ sup_{i,N} \E |(W_N)_{ii}|^2 <\infty,\\
& l_N(\epsilon)\to 0, \ \text{as} \ N\to \infty, \ \forall \epsilon>0, \ \text{where} \\
& l_N(\epsilon)= \frac{1}{N}\* \sum_{1\leq i \leq N} \E\left( |(W_N)_{ii}|^2 \*\chi (|(W_N)_{ii}|\geq \epsilon\*\sqrt{N})\right ).
\end{align*}
In the Hermitian case, one requires that, in addition, $\Re (W_N)_{ij}$ is independent from $\Im (W_N)_{ij}, \ \ 1\leq i<j\leq N,$ and
\begin{equation*}
\V(\Re(W_N)_{ij})=\V(\Im(W_N)_{ij})=\frac{\sigma^2}{2}, \ 1\leq i<j\leq N.
\end{equation*}
\end{remark}


\section{ \bf{Mathematical Expectation and Variance of Resolvent Entries}}
\label{sec:expvar}
In this section, we estimate mathematical expectation and variance of resolvent entries $R_{ij}(z):=(R_N(z))_{ij}.$ 
Without loss of generality, we can restrict our attention to the real symmetric case. The proofs in the Hermitian case are very similar.
Usually, we will assume in our calculations that $\mu_1=\mu.$  The proofs in the case $\mu_1\not=\mu$  are very similar.
From time to time, we will point out the (small) changes in the proofs one needs to make if  $\mu_1\not=\mu.$

\begin{proposition}
\label{proposition:prop1}
Let $X_N=\frac{1}{\sqrt{N}} W_N$ be a random real symmetric (Hermitian) Wigner matrix defined in (\ref{offdiagreal1}-\ref{diagreal2})
((\ref{offdiagherm1}-\ref{diagherm1}))  and $R_N(z)=(z\*I_N-X_N)^{-1} $ where $z \in \C. $
We will denote by $P_l(x), \ l\geq 1, $ a polynomial of degree $l$ with fixed positive coefficients.

Then 
\begin{align}
\label{odinnadtsat100}
& \E \tr_N R_N = \E R_{ii}(z) =g_{\sigma}(z) + O \left( \frac{1}{|\Im z|^6 \*N}\right),  \\
& \text{uniformly on bounded subsets of} \  \C\setminus \R,  \nonumber \\
\label{odinnadtsat101}
& \E R_{ij}(z)=O \left( \frac{P_5(|\Im z|^{-1})}{N}\right), \  1\leq i\not=j \leq N,\  \text{uniformly on} \ \C\setminus \R, \\
\label{odinnadtsat102}
& \V R_{ij}(z) = O \left( \frac{P_6(|\Im z|^{-1})}{N}\right), \  1\leq i, j\leq N, \ \text{uniformly on} \ \C\setminus \R. 
\end{align}

In addition, if $\mu$ has finite fifth moment and $\mu_1$ has finite third moment, then
\begin{equation}
\label{odinnadtsat103}
\E R_{ij}(z)=O \left( \frac{P_9(|\Im z|^{-1})}{N^{3/2}}\right), \  1\leq i\not=j \leq N, \ \text{uniformly on} \ \C\setminus \R.
\end{equation}
\end{proposition}

\begin{remark}
We refer the reader to \cite{EYY} (see e.g. Theorem 2.1 there) 
for the optimal (up to $\log N$ factors) estimates on the resolvent entries with the correct $|\Im z|^{-1}$ behavior. 
The authors in \cite{EYY} require that the marginal distributions are 
subexponential. For related results, we refer the reader to the survey \cite{E}.
\end{remark}

\begin{proof}
We use the resolvent identity (\ref{resident}) to write
\begin{equation}
\label{odin}
z\* \E R_{ij}(z)  = \delta_{ij}+\sum_{k} \*\E(X_{ik}\*R_{kj}).
\end{equation}
Applying the decoupling formula (\ref{decouple}) to the term $\E(X_{ik}\*R_{kj}(z))$ in (\ref{odin}), we obtain the equation
\begin{equation}
\label{dva}
z \* \E R_{ij}(z) = \delta_{ij} +\sigma^2 \* \E \left( R_{ij}(z)\*\tr_N R_N(z) \right) +
\frac{\sigma^2}{N} \E \left( (R_N(z)^2)_{ij}  \right) -\frac{2\*\sigma^2}{N}\*\left(\E(R_{ii}(z)\*R_{ij}(z))\right) +r_N,
\end{equation}
where $r_N$ contains the third cumulant term corresponding to $p=2$ in (\ref{decouple}) for $k\not=i$,
and the error terms due to the truncation of the decoupling formula (\ref{decouple}) at $p=2$ for $k\not=i$ and at $p=0$ for $k=i.$
We rewrite (\ref{dva}) as
\begin{align}
\label{dva100}
& z \* \E R_{ij}(z) = \delta_{ij} + \sigma^2 \*  \E R_{ij}(z) \* g_{\sigma}(z) + \sigma^2 \*\E R_{ij}(z) \* \left(\E \tr_N R_N(z)-g_{\sigma}(z)\right)\\
\label{dva101}
& + \sigma^2 \*\Cov (R_{ij}(z), \tr_N R_N(z)) + 
\frac{\sigma^2}{N} \E \left( (R_N(z)^2)_{ij} \right) -\frac{2\*\sigma^2}{N}\*\left(\E(R_{ii}(z)\*R_{ij}(z))\right)
+r_N.
\end{align}
We claim the following estimates uniformly on $\C\setminus \R.$
\begin{lemma}
\label{Lemma1}
\begin{align}
\label{2}
& \Cov(R_{ij}(z), \tr_N R_N(z))=O\left(\frac{P_3(|\Im z|^{-1})}{N}\right), \\
\label{3}
& \frac{\sigma^2}{N} \E \left( (R_N(z)^2)_{ij} \right) =O\left(\frac{P_2(|\Im z|^{-1})}{N}\right), \\
\label{4}
& r_N=O\left(\frac{P_4(|\Im z|^{-1})}{N}\right).
\end{align}

\end{lemma}
\begin{proof}
The bound (\ref{3}) immediately follows from (\ref{resbound}).
To obtain (\ref{2}), we again use (\ref{resbound}) and the estimate
\begin{equation}
\label{shcherb}
\V (\tr_N R_N(z))= O\left(\frac{1}{|\Im z|^{4}\*N^2}\right),
\end{equation}
from Proposition 2 of \cite{Sh}. It follows from the proof that the bound is valid provided the second
moments
of the diagonal entries are uniformly bounded and the fourth moments
of the off-diagonal entries are also uniformly bounded (\cite{Sh1}).

Now, we turn our attention to (\ref{4}).
First, we note that the term $\frac{2\*\sigma^2}{N}\*\left(\E(R_{ii}(z)\*R_{ij}(z))\right)$ in (\ref{dva101}) is $O\left(\frac{1}{|\Im z|^2 \* N} \right)$
which immediately follows from (\ref{resbound}).

The third cumulant term gives
\begin{align}
\label{tretii}
& \frac{\kappa_3}{2!\*N^{3/2}} [ 4\* \E (\sum_{k:k\not=i} R_{ij}\*R_{ik}\*R_{kk}) + 
2\*\E (\sum_{k:k\not=i} R_{ii}\*R_{kk}\*R_{kj}) \\
&+ 2\*\E (\sum_{k:k\not=i} (R_{ik})^2 \* R_{jk}) ].\nonumber
\end{align}
Since 
\begin{equation}
\label{solnce}
\sum_k |R_{ik}|^2 \leq \|R \|^2 \leq \frac{1}{|\Im z|^2}, \ \text{and} \ |R_{pq}|(z) \leq \frac{1}{|\Im z|},
\end{equation}
we conclude that the third cumulant term contributes $O\left(\frac{1}{N\*|\Im z|^3}\right) $ to $r_N$ in (\ref{dva101}).
In a similar way, the error term that appears due to the truncation of the decoupling formula (\ref{decouple}) at $p=2 $ is 
$O\left(\frac{1}{N\*|\Im z|^4}\right). $  Indeed, it can be written as a sum of $O(N)$ terms, where each term is bounded by
$O\left(\frac{\kappa_4}{N^2} \* |\Im z|^{-4}\right). \ $ 
Lemma \ref{Lemma1} is proven.
\end{proof}

{\it Proof of (\ref{odinnadtsat100})}

Now, we turn our attention to (\ref{odinnadtsat100}). For 
\begin{equation}
g_N(z):=\E \tr_N R= \E R_{11} \nonumber
\end{equation}
one can write the Master Equation as
\begin{equation}
\label{5}
z\* g_N(z)= 1+ \sigma^2 \*g_N^2(z) + \sigma^2\*\Cov (R_{11}(z), \tr_N R_N(z)) + \frac{\sigma^2}{N} \E \left( (R_N(z)^2)_{11} \right) 
+ r_N, 
\end{equation}
by applying (\ref{decouple}) to $\E(X_{1k}\*R_{k1}(z))$ and summing over $ 1\leq k \leq N. $
As before, $r_N$ contains the third cumulant term corresponding to $p=2$ in (\ref{decouple}) for $k\not=i$,
and the error due to the truncation of the decoupling formula (\ref{decouple}) at $p=2$ for $k\not=i$ and at $p=0$ for $k=i.$

By (\ref{2}) and (\ref{3}), we bound the third and the fourth terms on the r.h.s. of (\ref{5}) by
$O\left(\frac{P_3(|\Im z|^{-1})}{N}\right)$ and $O\left(\frac{P_2(|\Im z|^{-1})}{N}\right),$
respectively.
Thus, we obtain
\begin{equation}
\label{goodest}
z\* g_N(z)= 1+ \sigma^2 \*g_N^2(z) + O\left(\frac{P_4(|\Im z|^{-1})}{N}\right),
\end{equation}
uniformly in $z\in \C\setminus \R.$

We now show that the bound (\ref{goodest}) implies (\ref{odinnadtsat100})
uniformly in $z$ satisfying
\begin{equation}
\label{tiho17}
|z|\leq T, \ \ \text{and} \ \Im z \not=0,
\end{equation}
where $T$ is an arbitrary large fixed positive number. 
Our proof follows closely arguments from \cite{CDF} Proposition 4.2, \cite{CD} Section 3.4, 
and \cite{HT} Lemma 5.5, Proposition 5.6, and Theorem 5.7. 
Define 
\begin{equation}
\label{boloniya}
\mathcal{Q}_N=\{z: |z|<T+1, \ |\Im z|> L\*N^{-1/5} \},
\end{equation}
where $L>0$ will be chosen to be sufficiently large.
Then for $z \in \mathcal{Q}_N, $ the error term $ \frac{1}{N\*|\Im z|^4}= O(N^{-1/5}). \ $
Therefore
\begin{equation}
\label{goodest18}
z\* g_N(z)-\sigma^2 \*g_N^2(z)= 1+ O(N^{-1/5}),
\end{equation}
so 
\begin{equation}
\label{sparta}
|g_N(z)|\geq \delta_1>0 \ \text{on} \ \mathcal{Q}_N. 
\end{equation}
Let
\begin{equation*}
s_N(z)= \sigma^2\*g_N(z) +\frac{1}{g_N(z)}.
\end{equation*}
Then it follows from (\ref{goodest18}) and (\ref{sparta})
\begin{equation}
\label{novoeutro}
s_N(z)-z=\sigma^2\*g_N(z) +\frac{1}{g_N(z)} -z= O(\frac{1}{N\*|\Im z|^4})
\end{equation}
on $\mathcal{Q}_N.\ $ Since
\begin{equation*}
\sigma^2\*g_N(z) +\frac{1}{g_N(z)}=\sigma^2\*g_{\sigma}(s_N(z)) +\frac{1}{g_{\sigma}(s_N(z))},
\end{equation*}
we conclude that 
\begin{equation}
\label{marta}
g_N(z)=g_{\sigma}(s_N(z)),
\end{equation}
first for $|\Im z|> \sqrt{2}\*\sigma, $ and then for all $z \in \mathcal{Q}_N $ by the principle of uniqueness of analytic continuation.

Choosing $L$ in (\ref{boloniya}) sufficiently large, we have that 
\begin{equation*}
|\Im s_N(z)| \geq \frac{1}{2}\*|\Im z |
\end{equation*}
on $\mathcal{Q}_N.\ $  Since $|\frac{d\*g_{\sigma}(z)}{dz}|\leq \frac{1}{|\Im z|^2}, \ $ 
we conclude that (\ref{novoeutro}) and (\ref{marta}) imply (\ref{odinnadtsat100}) on $\mathcal{Q}_N.\ $

If $ |\Im z|\leq L\*N^{-1/5}, $ then $\frac{1}{N\*|\Im z|^5} \geq L^{-5}, \ $ and 
\begin{equation*}
|g_N(z)-g_{\sigma}(z)|\leq \frac{2}{|\Im z|}=O\left(\frac{1}{N\*|\Im z|^6}\right).
\end{equation*}
Therefore, the estimate (\ref{odinnadtsat100}) is proven.

{\it Proof of (\ref{odinnadtsat101})}

Now, we prove (\ref{odinnadtsat101}). It follows from (\ref{dva100}-\ref{dva101}) and Lemma \ref{Lemma1} that
\begin{equation}
\label{dvadvesti}
z\* \E R_{12} = \sigma^2\* g_N(z) \* \E R_{12}+
\frac{\sigma^2}{N} \E ((R_N(z))^2)_{12} + O\left( \frac{P_4(|\Im z|^{-1})}{N}\right).
\end{equation}
Therefore,
\begin{equation}
\label{basle}
\left(z-\sigma^2\*g_N(z)\right)\*\E R_{12}= O\left( \frac{P_4(|\Im z|^{-1})}{N}\right).
\end{equation}
It follows from (\ref{goodest}) that
\begin{equation}
\label{vasil}
g_N(z)\*(z-\sigma^2\*g_N(z))=1+ O\left( \frac{P_4(|\Im z|^{-1})}{N}\right).
\end{equation}
It follows from (\ref{basle}) and (\ref{vasil}) that
\begin{equation}
\label{alushta}
\left(1+ O\left( \frac{P_4(|\Im z|^{-1})}{N}\right) \right) \* \E R_{12}= O\left( \frac{P_4(|\Im z|^{-1})}{N}\right) \* g_N(z).
\end{equation}
Consider $\mathcal{O}_N= \{z: |\Im z|> L\*N^{-1/4} \},$
where the constant $L$ is chosen sufficiently large so that the 
$O\left( \frac{P_4(|\Im z|^{-1})}{N}\right)$ term on the l.h.s. of (\ref{alushta}) is at most $1/2$ in absolute value.
Since  $|g_N(z)|\leq \frac{1}{|\Im z|}, \ $  we obtain
\begin{equation}
\label{igkan}
|\E R_{12}| \leq \frac{1}{|\Im z|}\*  
O\left(\frac{P_4(|\Im z|^{-1})}{N}\right)=
O\left(\frac{P_5(|\Im z|^{-1})}{N}\right)
\end{equation}
for $z \in \mathcal{O}_N. \ $
On the other hand, if $|\Im z|\leq L\*N^{-1/4},$ then $ \frac{1}{N\*|\Im z|^4} \geq L^{-4}, $ and
\begin{equation}
\label{igkan1}
|\E R_{12}|\leq \frac{1}{|\Im z|}= O\left(\frac{1}{N\*|\Im z|^5}\right).
\end{equation}
Combining (\ref{igkan}) and (\ref{igkan1}),   we obtain (\ref{odinnadtsat101}).

{\it Proof of (\ref{odinnadtsat102})}

Now we proceed to prove the variance bound (\ref{odinnadtsat102}).
We apply the resolvent identity (\ref{resident}) to write
\begin{equation}
\label{odin101}
z\* \E \left(R_{ij}(z)\* R_{ij}(\bar{z})\right)  = 
\E R_{ij}(\bar{z})\*\delta_{ij} +\sum_{k} \*\E\left(X_{ik}\*R_{kj}\*R_{ij}(\bar{z})\right).
\end{equation}
Applying the decoupling formula (\ref{decouple}) to the term $\E\left(X_{ik}\*R_{kj}(z)\*R_{ij}(\bar{z})\right)$ in (\ref{odin101}), 
we obtain 
\begin{align}
\label{DVA}
& z \* \E \left(R_{ij}(z)\* R_{ij}(\bar{z}) \right) =\E R_{ij}(\bar{z})\*\delta_{ij} 
+\sigma^2 \* \E \left( R_{ij}(z)\*\tr_N R_N(z) \*R_{ij}(\bar{z})\right) \\
\label{DVA1}
& +\frac{\sigma^2}{N} \E \left( (R_N(z)^2)_{ij}\*R_{ij}(\bar{z})  \right) +\frac{\sigma^2}{N} \E\left(R_{ii}(\bar{z})\* (|R_N(z)|^2)_{jj} \right) \\
\label{DVA2}
& + \frac{\sigma^2}{N} \E\left( R_{ij}(\bar{z})\* (|R_N(z)|^2)_{ij}\right)+ r_N,
\end{align}
where as before $r_N$ contains the third cumulant term corresponding to $p=2$ in (\ref{decouple}) for $k\not=i$,
and the error terms due to the truncation of the decoupling formula (\ref{decouple}) at $p=2$ for $k\not=i$ and at $p=0$ for  $k=i.$
One can treat $r_N$ as before and obtain
\begin{equation}
\label{runran}
r_N= O\left(\frac{P_5(|\Im z|^{-1})}{N}\right),
\end{equation}
uniformly on $C\setminus \R.$
Indeed, the bound on the truncation error term at $p=2$ for $k\not=i$ follows from the fact that it can be written as a sum of $O(N)$ terms, where each 
term is bounded by
$ O\left(\frac{\kappa_4}{N^2}\* |\Im z|^{-5}\right). \ $ 
The truncation error term at $p=0$ for $k=i$ contains one term bounded by $ O\left(\frac{\kappa_2(\mu_1)}{N}\* |\Im z|^{-3}\right). \ $ 
The third cumulant term can be written as
\begin{equation}
\label{kalosha}
\frac{\kappa_3}{2!\*N^{3/2}}\* \E \left[ \sum_{k:k\not=i} \frac{\partial^2{R_{kj}(z)}}{\partial{X_{ik}^2}} \* R_{ij}(\bar{z})
+ 2\*\frac{\partial{R_{kj}(z)}}{\partial{X_{ik}}}\*\frac{\partial{R_{ij}(\bar{z})}}{\partial{X_{ik}}} + R_{kj}(z)\* 
\frac{\partial^2{R_{ij}(\bar{z})}}{\partial{X_{ik}^2}} \right].
\end{equation}
We will treat each of the three subsums in (\ref{kalosha}) separately.
The first one is equal to
\begin{align}
& \frac{\kappa_3}{2!\*N^{3/2}}\* \E [ \sum_{k:k\not=i} \frac{\partial^2{R_{kj}(z)}}{\partial{X_{ik}^2}} \* R_{ij}(\bar{z}) ] \nonumber \\
& =\frac{\kappa_3}{2!\*N^{3/2}}\* [ 4\* \E (\sum_{k:k\not=i} R_{ij}(z)\*R_{ik}(z)\*R_{kk}(z)\*R_{ij}(\bar{z}) ) + 
2\*\E (\sum_{k:k\not=i} R_{ii}(z)\*R_{kk}(z)\*R_{kj}(z)\*R_{ij}(\bar{z})) \nonumber \\
&+ 2\*\E (\sum_{k:k\not=i} (R_{ik}(z))^2 \* R_{jk}(z)\*R_{ij}(\bar{z})) ]  . \nonumber
\end{align}
The same arguments as after (\ref{tretii}) bound it by $O\left(\frac{1}{N \*|\Im z|^4}\right).$
The second subsum equals
\begin{align}
& \frac{\kappa_3}{N^{3/2}}\* \E \left[ \sum_{k:k\not=i} \* 
\frac{\partial{R_{kj}(z)}}{\partial{X_{ik}}}\*\frac{\partial{R_{ij}(\bar{z})}}{\partial{X_{ik}}}\right] \nonumber \\
& = \frac{\kappa_3}{N^{3/2}}\* \E \left[ \sum_{k:k\not=i} 
\left(R_{ij}(z)\*R_{kk}(z) + R_{ik}(z)\*R_{kj}(z)\right)\*\left(R_{ii}(\bar{z})\*R_{kj}(\bar{z}) +
R_{ji}(\bar{z})\*R_{ki}(\bar{z})\right) \right] \nonumber.
\end{align}
It follows from (\ref{solnce}) that the second subsum is $O(\frac{1}{N \*|\Im z|^4}).$
Finally, the third subsum equals
\begin{equation*}
\frac{\kappa_3}{2!\*N^{3/2}}\* \E \left[ \sum_{k:k\not=i} R_{kj}(z)\* \frac{\partial^2{R_{ij}(\bar{z})}}{\partial{X_{ik}^2}} \right]\leq 
\frac{\kappa_3}{2!\*N^{3/2}}\* O(|\Im z|^{-3}) \* \E \sum_{k} |R_{kj}(z)|,
\end{equation*}
so it is also $O\left(\frac{1}{N \*|\Im z|^4}\right).$

Using the bound (\ref{shcherb}) on the variance of $\tr_N R_N(z) $  and (\ref{odinnadtsat100}-\ref{odinnadtsat101}),
we estimate the last term in (\ref{DVA}) as
\begin{equation*}
\E \left( R_{ij}(z)\*\tr_N R_N(z) \*R_{ij}(\bar{z})\right)= g_N(z) \* \E \left( R_{ij}(z)\*R_{ij}(\bar{z})\right)\* 
+ O\left(\frac{1}{|\Im z|^4 \* N}\right),
\end{equation*}
where we recall that $g_N(z)= \E \tr_N R_N(z). $
Since the two terms in (\ref{DVA1}) and the first term in (\ref{DVA2}) are bounded by $O\left(\frac{1}{|\Im z|^3 \* N}\right), $
we conclude that
\begin{equation}
\label{eaton}
z \* \E \left(R_{ij}(z)\* R_{ij}(\bar{z}) \right)= \E R_{ij}(\bar{z})\* \delta_{ij}+\sigma^2 \*g_N(z) \*\E \left( R_{ij}(z)\*R_{ij}(\bar{z})\right)+
O\left(\frac{P_5(|\Im z|^{-1})}{N}\right)
\end{equation}
uniformly on $\C\setminus \R.$
We now rewrite (\ref{goodest}) and (\ref{dvadvesti})
as 
\begin{equation}
\label{goodest117}
z\* \E R_{ij}(z)= \delta_{ij}+ \sigma^2 \* g_N(z) \* E R_{ij}(z) + O\left(\frac{P_4(|\Im z|^{-1})}{N}\right).
\end{equation}
Multiplying both sides of (\ref{goodest117}) by $\E R_{ij}(\bar{z})$ and subtracting from (\ref{eaton}), we obtain
\begin{equation}
(z -\sigma^2\*g_N(z))\* V (R_{ij}(z))= O\left(\frac{P_5(|\Im z|^{-1})}{N}\right),
\end{equation}
uniformly on $\C\setminus \R. $
Repeating the arguments after (\ref{dvadvesti}), we conclude that
\begin{equation*}
\V R_{ij}(z) = O\left(\frac{P_6(|\Im z|^{-1})}{N}\right).
\end{equation*}
This is exactly (\ref{odinnadtsat102}).

{\it Proof of (\ref{odinnadtsat103})}

Now, we turn our attention to the proof of (\ref{odinnadtsat103}). 
Let us assume that $\mu$ has finite fifth moment, $\mu_1$ has finite third moment, and $i\not=j. $  Without loss of generality, we can assume 
$i=1$ and $j=2.$
We write the Master equation for $\E R_{12}(z)$ as
\begin{align}
\label{dvakur}
& z \* \E R_{12}(z) = \E \left( \sum_k X_{1k}\*R_{k2}(z) \right)= \sigma^2 \* \E \left( R_{12}(z)\*\tr_N R_N(z) \right) \\
\label{dvakura}
& + \frac{\sigma^2}{N} \E \left( (R_N(z)^2)_{12}  \right) +r_N, 
\end{align}
where we apply the decoupling formula (\ref{decouple}) to the term $\E(X_{1k}\*R_{2j}(z))$ and truncate it at $p=3$ for $k\not=1$ and at $p=1$
for $k=1.$
Thus, the $r_N$ term contains the third and fourth cumulant terms (corresponding to $p=2$ and $p=3$) for $k\not=1$ 
as well as the error terms due to the truncation of the decoupling 
formula at $p=3$ for $k\not=1$  and at $p=1$ for $k=1.$
We note that in order to truncate the decoupling formula at $p=3, $ we have to require that $\mu$ has finite fifth moment.

It follows from (\ref{odinnadtsat102}) and (\ref{shcherb}) that we can replace 
in (\ref{dvakur}) $\E \left( R_{12}(z)\*\tr_N R_N(z) \right)$  by $\E (R_{12}(z)) \* \E \tr_N R_N(z)$ up to the error
$ O\left(\frac{P_5(|\Im z|^{-1})}{N^{3/2}}\right).$
We bound the absolute values of the terms in
(\ref{dvakura}) by $O\left(\frac{P_8(|\Im z|^{-1})}{N^{3/2}}\right)$ (see Lemmas \ref{lemka1} and \ref{lemka2} below).
Combining these results, we obtain
\begin{equation}
\label{71}
z \* \E R_{12}(z) =\sigma^2\* g_N(z)\* \E R_{12}(z) + O\left(\frac{P_8(|\Im z|^{-1})}{N^{3/2}}\right).
\end{equation}
Repeating the arguments after (\ref{basle}), we obtain (\ref{odinnadtsat103}).


\begin{lemma}
\label{lemka1}
Let $\mu$ have finite fifth moment and $\mu_1$ has finite third moment.  Then
\begin{equation}
\label{72}
\E \left( (R_N(z)^2)_{12}  \right)= O\left(\frac{P_8(|\Im z|^{-1})}{N^{1/2}}\right),
\end{equation}
uniformly on $\C\setminus \R.$
\end{lemma}
\begin{proof}
We write the Master equation and use (\ref{odinnadtsat101}) to obtain
\begin{align}
& z \* \E (R^2)_{12}= z\*\E \sum_k R_{1k}\*R_{k2}= \E R_{12} + \E \sum_l \sum_k X_{1l}\*R_{lk}\*R_{k2} \nonumber \\
&  =\frac{\sigma^2}{N} \* \E [\sum_{l,k} (R_{ll}\*R_{1k} + R_{l1}\*R_{lk})\*R_{k2}]
+ \frac{\sigma^2}{N} \* \E [ \sum_{l,k} R_{lk} \*(R_{k1}\*R_{l2} + R_{kl}\*R_{12})]  \nonumber \\ 
& + O\left(\frac{P_5(|\Im z|^{-1})}{N}\right)  +r_N, \nonumber
\end{align}
where $r_N$ contains the third and fourth cumulant terms and the error terms due to the truncation in the decoupling formula at $p=3$ for 
$l\not=1$ and at $p=1$ for $l=1.$
Thus, we have
\begin{align}
& z \* \E (R^2)_{12}= \sigma^2 \* \E [ (R^2)_{12}\* \tr_N R] + 2\* \frac{\sigma^2}{N} \E (R^3)_{12} + \frac{\sigma^2}{N}\* \E [R_{12} \* \sum_{l,k}
R_{lk}^2 ] + O\left(\frac{P_5(|\Im z|^{-1})}{N}\right)  +r_N \nonumber \\
& = \sigma^2\*g_N(z) \* \E (R^2)_{12} + \frac{\sigma^2}{N}\* \E [R_{12} \* \sum_{l,k}
R_{lk}^2 ] + O\left(\frac{P_5(|\Im z|^{-1})}{N}\right)  +r_N,\nonumber
\end{align}
where we used (\ref{2}).
We note that
\begin{equation*}
\frac{\sigma^2}{N}\* \E R_{12} \* \E (\sum_{l,k} R_{lk}^2) =O\left(\frac{P_5(|\Im z|^{-1})}{N^2}\right)
\* O\left(\frac{N}{|\Im z|^2}\right)= O\left(\frac{P_7(|\Im z|^{-1})}{N}\right),
\end{equation*}
and using (\ref{odinnadtsat102}), we have
\begin{equation*}
\frac{\sigma^2}{N}\*[\V R_{12}]^{1/2}\* [\V(\sum_{l,k} R_{lk}^2)]^{1/2}= 
O\left(\frac{P_3(|\Im z|^{-1})}{N^{3/2}}\right)
\* O\left(\frac{N}{|\Im z|^2}\right)
=O\left(\frac{P_5(|\Im z|^{-1})}{N^{1/2}}\right).
\end{equation*}
Thus, we arrive at
\begin{equation}
\label{75}
(z-\sigma^2\*g_N(z)) \* \E (R^2)_{12}=O\left(\frac{P_7(|\Im z|^{-1})}{N^{1/2}}\right) +r_N.
\end{equation}
Rather long but straightforward calculations bound $r_N$ in (\ref{75}) by $O\left(\frac{P_6(|\Im z|^{-1})}{N^{1/2}} \right).$
We leave the details to the reader.
Therefore, we have
\begin{equation}
\label{76}
(z-\sigma^2\*g_N(z)) \* \E (R^2)_{12}=O\left(\frac{P_7(|\Im z|^{-1})}{N^{1/2}}\right).
\end{equation}
Now consider, as  before,
$\mathcal{O}_N= \{z: |\Im z|> L\*N^{-1/4} \},$
where the constant $L$ is chosen sufficiently large.
It follows from (\ref{vasil}) that
\begin{equation*}
|\E (R^2)_{12}| \leq \frac{1}{|\Im z|}\* O\left(\frac{P_7(|\Im z|^{-1})}{N^{1/2}}\right)
=O\left(\frac{P_8(|\Im z|^{-1})}{N^{1/2}}\right),
\end{equation*}
for $z \in \mathcal{O}_N. \ $
On the other hand, if $|\Im z|\leq L\*N^{-1/4},$ then $ \frac{1}{N^{1/2}\*|\Im z|^2} \geq L^{-2}, $ and
\begin{equation}
|\E (R^2)_{12}|\leq \frac{1}{|\Im z|^2}= O\left(\frac{1}{N^{1/2}\*|\Im z|^4}\right),
\end{equation}
for $z \not\in \mathcal{O}_N. \ $
Lemma \ref{lemka1} is proven.
\end{proof}


\begin{lemma}
\label{lemka2}
The term $r_N$ on the r.h.s.
of (\ref{dvakura}) satisfies 
\begin{equation}
\label{73}
O\left(\frac{P_8(|\Im z|^{-1})}{N^{3/2}}\right).
\end{equation}
\end{lemma}
\begin{proof}
First, we look at the third cumulant terms
\begin{equation}
\label{tretii17}
\frac{\kappa_3}{2!\*N^{3/2}} [ 4\* \E (\sum_{k\not=1} R_{12}\*R_{1k}\*R_{kk}) + 
2\*\E (\sum_{k\not=1} R_{11}\*R_{kk}\*R_{k2}) 
+ 2\*\E (\sum_{k\not=1} (R_{1k})^2 \* R_{2k})].
\end{equation}
To estimate the first subsum in (\ref{tretii17}), we write
\begin{align}
& | \E [R_{12} \* (\sum_{k\not=1} \*R_{1k}\*R_{kk})]- \E [R_{12}]\* \E [\sum_{k\not=1} \*R_{1k}\*R_{kk}] | \leq \left(\V(R_{12})\right)^{1/2} \* 
\left( \V( \sum_{k\not=1} \*R_{1k}\*R_{kk})\right)^{1/2} \nonumber \\
& \leq O\left(\frac{P_3(|\Im z|^{-1})}{N^{1/2}} \right) 
\* \frac{N^{1/2}}{|\Im z|^{2}} \leq 
O(P_5(|\Im z|^{-1})). \nonumber
\end{align}
Taking into account (\ref{odinnadtsat101}) and (\ref{solnce}), we have
\begin{equation*}
\E [R_{12}]\* \E [\sum_{k\not=1} \*R_{1k}\*R_{kk}]= O\left(\frac{1}{|\Im z|^5\*N}\right) \* \frac{N^{1/2}}{|\Im z|^2}= 
O\left(\frac{P_7(|\Im z|^{-1})}{N^{1/2}}\right).
\end{equation*}
Therefore, we can bound the first subsum in (\ref{tretii17}) by
$O\left(\frac{P_7(|\Im z|^{-1})}{N^{3/2}}\right).$
To bound the second subsum in (\ref{tretii17}), we note that
\begin{equation}
\label{80}
\sum_{k\not=1} (\E R_{11}\*R_{kk}) \* \E R_{k2} = O(P_7(|\Im z|^{-1}))
\end{equation}
by (\ref{odinnadtsat101}) and (\ref{solnce}).
To bound
\begin{equation}
\label{81}
\sum_{k\not=1} (\V(R_{11}\*R_{kk}))^{1/2}\* (\V R_{k2})^{1/2},
\end{equation}
we use (\ref{odinnadtsat102}) and 
\begin{align}
& \V(R_{11}\*R_{kk}) \leq \E \left(|[(R_{11}-g_N(z)) +g_N(z)] \* [(R_{kk}-g_N(z)) +g_N(z)] - g^2_N(z)|^2\right) \nonumber \\
& = \E |(R_{11}-g_N(z))\*(R_{kk}-g_N(z))+ g_N(z)\*(R_{kk}-g_N(z)) +  g_N(z)\*(R_{11}-g_N(z))|^2. \nonumber
\end{align}
Using (\ref{odinnadtsat102}), we see that 
\begin{equation}
\label{82}
\V(R_{11}\*R_{kk}) = O\left(\frac{P_8(|\Im z|^{-1})}{N}\right).
\end{equation}
The bounds (\ref{82}), (\ref{odinnadtsat102}), and (\ref{80}) then show that the second subsum in (\ref{tretii17}) is 
$O\left(\frac{P_7(|\Im z|^{-1})}{N^{3/2}}\right).$

Finally, we bound the third subsum in (\ref{tretii17}) by using the estimate
\begin{equation*}
|\sum_{k\not=1x} (R_{1k})^2 \* R_{2k}|\leq |\Im z|^{-3}.
\end{equation*}

Now, let us look at the fourth cumulant terms:
\begin{align}
\label{chetvert}
& \frac{\kappa_4}{3!\*N^2} \* [ 18\* \E (\sum_{k\not=1} R_{11}\*R_{1k}\*R_{kk}\*R_{k2}) +  
6 \*\E (\sum_{k\not=1} R_{11}\*(R_{kk})^2\*R_{12}) \\
\label{chetchet}
& + 18\*\E (\sum_{k\not=1} (R_{1k})^2\*R_{kk}\*R_{12})  + 6 \* \E (\sum_{k} (R_{1k})^3\*R_{k2}) ]. 
\end{align}
Clearly,
\begin{align}
& |\sum_{k\not=1} R_{11}\*R_{1k}\*R_{kk}\*R_{k2}|\leq |\Im z|^{-4}, \nonumber \\
& |\sum_{k\not=1} (R_{1k})^2\*R_{kk}\*R_{12}|\leq |\Im z|^{-4}, \ \text{and} \nonumber \\
& |\sum_{k\not=1} (R_{1k})^3\*R_{k2}| \leq |\Im z|^{-4}. \nonumber
\end{align}
To estimate the term
$ \frac{\kappa_4}{N^2}\* \E (\sum_{k\not=1} R_{11}\*(R_{kk})^2\*R_{12}),$ we note that by using (\ref{odinnadtsat101}-\ref{odinnadtsat102})
\begin{equation*}
\sum_{k\not=1} \E(R_{11}\*(R_{kk})^2)\* \E R_{12}= O(P_8(|\Im z|^{-1})).
\end{equation*}
We are left with estimating
\begin{equation*}
\sum_{k\not=1} \left(\V(R_{11}\*(R_{kk})^2)\right)^{1/2} \* \left(\V R_{12}\right)^{1/2}\leq N\* |\Im z|^{-3} \* \left(\V R_{12}\right)^{1/2}=
O(N^{1/2} \* P_6(|\Im z|^{-1}).
\end{equation*}
Combining the estimates of all fourth cumulant terms, we observe that the sums in (\ref{chetvert}-\ref{chetchet}) are bounded by 
$O\left(\frac{P_8(|\Im z|^{-6})}{N^{3/2}}\right).$

To bound the error term, we note that it contains $O(N)$ terms, such that each of them is at most $O\left(\frac{1}{|\Im z|^5\*N^{5/2}}\right).$
Thus, the error term is bounded by $O\left(\frac{1}{|\Im z|^5\*N^{3/2}}\right).$
Lemma \ref{lemka2} is proven.
\end{proof}
This finishes the proof of Proposition \ref{proposition:prop1}.
\end{proof}


\section{ \bf{Expectation and Variance of Matrix Entries Of Regular Functions of Wigner Matrices}}
\label{sec:expvarreg}
In this section, we estimate the mathematical expectation and the variance of matrix entries $f(X_N)_{ij}$ for regular test functions $f$. 
As before, without loss of generality, we can restrict our attention to the real symmetric case. The proofs in the Hermitian case are very similar.

To extend the results of Proposition \ref{proposition:prop1} to the case of more general test functions, we will exploit
the Helffer- Sj\"{o}strand functional calculus 
(see \cite{HS}, \cite{D}) that depends on the use of almost analytic extensions of functions due to H\"{o}rmander \cite{H1}, \cite{H2}.

We start by proving Proposition \ref{proposition:prop2}.
\begin{proof}
Let us first assume that $f$ has compact support
and prove (\ref{chto1}).
Using the Helffer-Sj\"{o}strand functional calculus 
(see \cite{HS}, \cite{D}), we can write for any self-adjoint operator $X$
\begin{equation}
 f(X)=-\frac{1}{\pi}\,\int_{\mathbb{C}}\frac{\partial \tilde{f}}{\partial \bar{z}}\, \frac{1}{z-X}\,dxdy 
\quad,\quad\frac{\partial \tilde{f}}{\partial \bar{z}} := \frac{1}{2}\Big(\frac{\partial \tilde{f}}
{\partial x}+i\frac{\partial \tilde{f}}{\partial y}\Big)
\label{formula-H/S}
 \end{equation}
 where:
 \begin{itemize}
\item[i)]
 $z=x+iy$ with $x,y \in \mathbb{R}$;
 \item[ii)] $\tilde{f}(z)$ is the extension of the function $f$ defined as follows
  \begin{equation}\label{a.a. -extension}
  \tilde{f}(z):=\Big(\,\sum_{n=0}^{l}\frac{f^{(n)}(x)(iy)^n}{n!}\,\Big)\sigma(y);
\end{equation}
 here $\sigma \in C^{\infty}(\mathbb{R})$ is a nonnegative function equal to $1$ for $|y|\leq 1/2$ and equal to zero for $|y|\geq 1$.
 \end{itemize}
It should be noted that the r.h.s. of (\ref{formula-H/S}) does not depend on the choice of $l$ and the cut-off function $\sigma(y) $ in 
(\ref{a.a. -extension})(see e.g. Theorem 2 in \cite{D}). 
For $X=X_N$,  (\ref{formula-H/S}) implies
 \begin{equation}
\label{integral1}
f(X_N)_{ii}=-\frac{1}{\pi}\,\int_{\mathbb{C}}\frac{\partial \tilde{f}}{\partial \bar{z}}  
R_{ii}(z) dxdy \end{equation}
It follows from  (\ref{odinnadtsat100}) that with $l=6$ in (\ref{a.a. -extension})
\begin{eqnarray}\label{zato}
& &\mathbb{E} f(X_N)_{ii}=-\mathbb{E} \frac{1}{\*\pi\* } \* \int_{\mathbb{C}} \frac{\partial \tilde{f}}{\partial \bar{z}}  \* 
R_{ii}(z) \*
dxdy\\
&= &
-\frac{1}{\pi}\, \*\,\int_{\mathbb{C}}\frac{\partial \tilde{f}}{\partial \bar{z}}\,   \* g_{\sigma}(z) \* dxdy
-\frac{1}{\pi}\, \,\int_{\mathbb{C}}\frac{\partial \tilde{f}}{\partial \bar{z}}\,  \* \epsilon_{ii}(z) \* dxdy \label{main equation}\\
&= &
\int f(x) d\mu_{sc}(x)
-\frac{1}{\pi}\, \,\int_{\mathbb{C}}\frac{\partial \tilde{f}}{\partial \bar{z}}\,  \* \epsilon_{ii}(z) \* dxdy\end{eqnarray}
where 
\begin{equation*}
| \epsilon_{ii}(z)|= | \E R_{ii}(z)- g_{\sigma}(z)| \leq C_1\left( \frac{1}{N}\,\frac{1}{|Im(z)|^6}\right), 
\end{equation*}
and $C_1$ depends on $\mu, \ \mu_1$ and $L,$ where $supp(f)\subset [-L, L].$

Using the definition of $\tilde{f}$ (see (\ref{a.a. -extension})) one can easily calculate 
\begin{eqnarray}
\frac{\partial \tilde{f}}{\partial \bar{z}}&=&\frac{1}{2}\Big(\frac{\partial \tilde{f}}{\partial x}+i\frac{\partial \tilde{f}}{\partial y}
\Big)\\
& =&\frac{1}{2} \Big(\,\sum_{n=0}^{6}\frac{f^{(n)}(x)(iy)^n}{n!}\,\Big)
i\frac{d\sigma}{dy}+\frac{1}{2}f^{(7)}(x)(iy)^6\frac{\sigma(y)}{6!}
\end{eqnarray}
and derive the crucial bound
\begin{equation}\label{estimate-derivative}
\Big|\frac{\partial \tilde{f}}{\partial \bar{z}}\Big| \leq const \* \|f\|_{C^7([-L, +L])}\*|y|^6,
\end{equation}
for $f \in C^7_c(\R)$ with $supp(f) \subset [-L, L]. $
Therefore, the second term on the r.h.s. of (\ref{main equation}) can be estimated as follows
\begin{eqnarray}
& &|\frac{1}{\pi}\, \,\int_{\mathbb{C}}\frac{\partial \tilde{f}}{\partial \bar{z}}\,  \* \epsilon_{ii}(z) \* dxdy| \\
&\leq  & \frac{1}{\pi}\, \,\int_{\mathbb{C}}|\frac{\partial \tilde{f}}{\partial \bar{z}}\,  \* \epsilon_{ii}(z) \*| dxdy\\
& \leq&C_1\*  const(L, \mu, \mu_1) \* \|f\|_{C^7([-L, +L])} \*\frac{1}{N} 
\int dx \chi_f(x) \int dy \chi_\sigma(y)
\end{eqnarray}
where $\chi_f$  and $\chi_{\sigma}$ are the characteristic functions of the support of $f$ and of $\sigma$ respectively.
This proves (\ref{chto1}) for $f \in C_c^7(\R). \ $ 

To prove (\ref{chto11}), one has 
to generalize the estimate (\ref{odinnadtsat100}) in Proposition \ref{proposition:prop1} to the whole complex plane.
We claim the following bound 
\begin{equation}
\label{30}
|\E R_{ii}(z)- g_{\sigma}(z)|=|\E \tr_N R_N - g_{\sigma}(z)| \leq  (|z|+M)\*\frac{P_7(|\Im z|^{-1})}{N},
\end{equation}
uniformly in $ z\in \C, \ \Im z \not=0, $ where $M$ is some positive constant.
The bound (\ref{30}) follows from (\ref{goodest}).  The proof is identical to the proof of similar bounds given in 
\cite{CDF} Proposition 4.2 and \cite{CD} Section 3.4.
Using the Helffer Sj\"{o}strand functional calculus as before, one proves (\ref{chto11}) provided $f$ has eight continuous derivatives and
$(|x|+1)\* \frac{d^lf}{dx^l}(x), \ 0\leq l \leq 8, $ are integrable on $\R. $

In the case of the off-diagonal entries,  one takes $l=5$ in (\ref{a.a. -extension}), so
\begin{equation}\label{a.a. -ext}
  \tilde{f}(z):=\Big(\,\sum_{n=0}^{5}\frac{f^{(n)}(x)(iy)^n}{n!}\,\Big)\sigma(y),
\end{equation}
and proceeds in a similar fashion.  The only significant difference is that one has to replace
the estimate (\ref{estimate-derivative}) by
\begin{equation}\label{estimate-derivative1}
\Big|\frac{\partial \tilde{f}}{\partial \bar{z}}(x+iy)\Big| \leq const \* \max\left(|\frac{d^lf}{dx^l}(x)|, \ 1\leq l \leq 6\right) \*|y|^5,
\end{equation}
which together with (\ref{odinnadtsat101}) implies (\ref{chto2}).  
To prove (\ref{chto4}), one takes $l=9$ in (\ref{a.a. -extension}) and uses 
(\ref{odinnadtsat103}).

To bound the variance, we write  
$ f(X_N)_{ij} $ using (\ref{integral1}) with
\begin{equation}\label{a.a. -exten}
  \tilde{f}(z):=\Big(\,\sum_{n=0}^{3}\frac{f^{(n)}(x)(iy)^n}{n!}\,\Big)\sigma(y).
\end{equation}
Then
\begin{eqnarray}
\frac{\partial \tilde{f}}{\partial \bar{z}}&=&\frac{1}{2}\Big(\frac{\partial \tilde{f}}{\partial x}+i\frac{\partial \tilde{f}}{\partial y}
\Big)\\
\label{19}
& =&\frac{1}{2} \Big(\,\sum_{n=0}^{3}\frac{f^{(n)}(x)(iy)^n}{n!}\,\Big)
i\frac{d\sigma}{dy}+\frac{1}{2}f^{(4)}(x)(iy)^3\frac{\sigma(y)}{3!}
\end{eqnarray}
and, in particular,
\begin{equation}\label{estimate-derivative3}
\Big|\frac{\partial \tilde{f}}{\partial \bar{z}} (x+iy)\Big|\leq  Const \* \max\left(|\frac{d^lf}{dx^l}(x)|, \ 1\leq l \leq 4\right) \*
|y|^3\quad.
\end{equation}

Now we are ready to bound $\V (f(X_N)_{ij}).$  We write
\begin{equation}
\label{varvara}
\V (f(X_N)_{ij})= \V \left( -\frac{1}{\pi}\,\int_{\mathbb{C}}\frac{\partial \tilde{f}}{\partial \bar{z}} R_{ij}(z) dxdy \right)=
\frac{1}{\pi^2} \,\int_{\mathbb{C}}\,\int_{\mathbb{C}} \frac{\partial \tilde{f}}{\partial \bar{z}} \*
\frac{\partial \tilde{f}}{\partial \bar{w}} \Cov (R_{ij}(z), R_{ij}(w)) \*dxdydudv,
\end{equation}
where $z=x+iy, \ w=u+iv.$
We then obtain the following upper bound from (\ref{odinnadtsat102})
\begin{align}
\label{17}
& \V (f(X_N)_{ij}) \leq \frac{1}{\pi^2}\,\int_{\mathbb{C}}\,\int_{\mathbb{C}}|\frac{\partial \tilde{f}}{\partial \bar{z}}|\* 
|\frac{\partial \tilde{f}}{\partial \bar{w}}| \* \sqrt{\V(R_{ij}(z))}\* \sqrt{\V(R_{ij}(w))}  dxdydudv \\
\label{18}
&  \leq \frac{Const}{N} \* \left(\int_{\mathbb{C}}\, |\frac{\partial \tilde{f}}{\partial \bar{z}}|\*  \frac{1}{|\Im z|^{3}}| dxdy \right)^2.
\end{align}
Plugging (\ref{a.a. -exten}) in (\ref{18}) and using (\ref{19}), we obtain (\ref{chto3}). 
Proposition \ref{proposition:prop2} is proven.
\end{proof}

If $\mu$ satisfies the Poincar\'e inequality (\ref{poin}), one can generalize the results of  
Proposition \ref{proposition:prop2}. Recall that we defined
$|f|_{\mathcal{L},\R} $ and $|f|_{\mathcal{L},\delta} $ in (\ref{snova}-\ref{snova1}).

\begin{proposition}
\label{proposition:prop4}
Let $X_N=\frac{1}{\sqrt{N}} W_N$ be a random real symmetric (Hermitian) 
Wigner matrix with the marginal distributions $\mu $ and $\mu_1$ of the matrix entries satisfying the
Poincar\'e inequality (\ref{poin}) and 
$f:\R \to  \R $ be a Lipschitz continuous function on 
$[-2\*\sigma -\delta, 2\*\sigma +\delta] $ for some $\delta>0.\ $  
Let us assume $f$ satisfy the subexponential growth condition (\ref{exprost}).
Then
\begin{align}
\label{kurykur}
&  \P \left( |f(X_N)_{ij}- \E(f(X_N)_{ij})| \geq t \right) \\
& \leq  2\*K \* \exp\left( -\frac{\sqrt{\upsilon \*N} \* t}{2 \*|f|_{\mathcal{L},\delta}} \right) + (2\*K+ o(1))\* 
\exp\left(-\frac{\sqrt{\upsilon\*N}\*\delta}{2}\right),
\nonumber
\end{align}
where $|f|_{\mathcal{L},\delta} $ is defined in (\ref{snova1}), $K$ is defined in (\ref{KK}), and
$\upsilon $ is the constant in the Poincar\'e inequality (\ref{poin}).
If, in addition, $f \in C^7(\R)$ for $i=j $ ( $ f\in C^6(\R)$ for $i\not=j $),
then
\begin{equation}
\label{47}
\E (f(X_N)_{ij})=\delta_{ij}\* \int_{-2\sigma}^{2\sigma} f(x) \* \frac{1}{2 \pi \sigma^2} \sqrt{ 4 \sigma^2 - x^2}  
\* dx + O\left(\frac{1}{N}\right).
\end{equation}

If the marginal distributions $\mu $ and $\mu_1$ satisfy the
Poincar\'e inequality (\ref{poin}) and $f$ is a Lipschitz continuous function on $\R,$
then
\begin{align}
\label{kurykur1}
&  \P \left( |f(X_N)_{ij}- \E(f(X_N)_{ij})| \geq t \right) \\
& \leq  2\*K \* \exp\left( -\frac{\sqrt{\upsilon \*N} \* t}{2\*|f|_{\mathcal{L}, \R}} \right), \nonumber
\end{align}
where $|f|_{\mathcal{L},\R} $ is defined in (\ref{snova}).
\end{proposition}

\begin{proof}
Let us assume that $\mu $ and $\mu_1$ satisfy the
Poincar\'e inequality (\ref{poin}).  Suppose that $f(x)$ is a Lipschitz continuous function on $\R$ with the Lipschitz constant $|f|_{\mathcal{L}, \R}.$
Then
the matrix-valued function $f(X)$ on the space of $N\times N$ real symmetric (Hermitian) matrices is also Lipschitz continuous with respect to the
Hilbert-Schmidt norm (\cite{CQT}, Proposition 4.6, c)) with the same Lipschitz constant.  Namely,
\begin{equation}
\label{villa}
\|f(X)-f(Y)\|_{HS}\leq |f|_{\mathcal{L},\R} \* \|X-Y\|_{HS},
\end{equation}
where 
\begin{equation}
\|X-Y\|_{HS}=\left(\Tr (|X-Y|^2)\right)^{1/2}.
\end{equation}
(We note that even though (\ref{villa}) was proven in \cite{CQT} only for real symmetric matrices, the proof for Hermitian matrices
is essentially the same).  Therefore, $f(X)_{ij}$ is a Lipschitz continuous function of the matrix entries of $X$ with the same Lipschitz constant.
Since the Poincar\'e inequality tensorizes (\cite{GZ}, \cite{AGZ}), the joint distribution of the matrix entries 
$ \{ X_{ii}, \ 1\leq i \leq N, \ \sqrt{2}\*X_{jk}, \ 1\leq j<k\leq N \} $ of $X_N$ satisfies the Poincar\'e inequality
with the constant $\frac{1}{2}\*N\*\upsilon.$
Therefore, for any real-valued Lipschitz continuous function of the matrix entries with the Lipschitz constant
\begin{equation*}
|G|_{\mathcal{L}}:=sup_{X\not=Y} \frac{|G(X)-G(Y)|}{\|X-Y\|_{HS}},
\end{equation*}
one has (see e.g. \cite{AGZ}, Lemma 4.4.3 and Exercise 4.4.5)
\begin{equation}
\label{BOLUKLON}
\P\left(|G(X_N)-\E G(X_N)|\geq t \right) \leq 2\*K\* \exp\left(-\frac{\sqrt{\upsilon\*N}}{2\*|G|_{\mathcal{L}}}\*t\right),
\end{equation}
with
\begin{equation*}
K=-\sum_{i\geq 0} 2^i\*\log(1-2^{-1}\*4^{-i}).
\end{equation*}
Applying (\ref{BOLUKLON}) to $f(X_N)_{ij}$ one obtains (\ref{kurykur1}).

Now let us relax our assumptions on $f$ and consider
$f:\R \to  \R $ Lipschitz continuous function on 
$[-2\*\sigma -\delta, 2\*\sigma +\delta] $ for some $\delta>0\ $  that satisfies the
subexponential growth condition (\ref{exprost}).  Let $h(x)$ be a $C^{\infty}(\R)$ function with compact support that is identically one in the 
neighborhood of the support of the Wigner semicircle law, i.e.,
\begin{equation}
\label{hhh}
h(x)\equiv 1 \ \text{ for} \ x \in [-2\sigma -\delta, 2\sigma +\delta], \ h\in C_c^{\infty}(\R).
\end{equation}
For non-constant $f$, we can always choose $h$ in such a way that 
\begin{equation}
\label{amerika1}
|hf|_{\mathcal{L},\R}=|f|_{\mathcal{L},\delta}
\end{equation}
Note that 
\begin{equation}
\label{egalite}
f(X_N)_{ij}= (fh)(X_N)_{ij} \ \text{when} \ \|X_N\|\leq 2\*\sigma+\delta.
\end{equation}
It follows from the universality results on the distribution of 
the largest eigenvalues of $X_N$ (see \cite{J} and also \cite{EYY}, \cite{TV}, \cite{PS1}, \cite{PS2}, 
\cite{S}, \cite{TW1}, \cite{TW2}) that 
\begin{equation*}
\|X_N\| =2\*\sigma +O(N^{-1/2-1/100})
\end{equation*}
with probability going to 1.
Moreover, $G(X)=\|X\|$ is a Lipschitz continuous function of the matrix entries with Lipschitz constant one.  Thus,
(\ref{BOLUKLON}) implies
\begin{equation}
\label{bol}
\P\left(| \|X_N\|- 2\*\sigma|\geq t \right) \leq (2\*K +o(1))\* \exp\left(-\frac{\sqrt{\upsilon\*N}}{2}\*t\right).
\end{equation}
In particular,
\begin{equation}
\label{bolukl}
\P (\|X_N\|> 2\*\sigma +\delta) \leq (2\*K+o(1))\* \exp\left(-\frac{\sqrt{\upsilon\*N}\*\delta}{2}\right).
\end{equation}
Then the estimate (\ref{kurykur}) for $f(X_N)_{ij} $ follows from the estimate (\ref{kurykur1}) for $(hf)(X_N)_{ij}, $  (\ref{egalite}), and (\ref{bolukl}).
Finally, the estimate (\ref{47}) follows from (\ref{chto1}) and (\ref{chto2}) for $(hf)(X_N)_{ij}, $  (\ref{exprost}), and  (\ref{bol}).
\end{proof}
\begin{remark}
Let $f^{(z)}(x)=\frac{1}{z-x}, $ where $z \not\in [-2\sigma-\delta, 2\sigma+\delta]$  for some $\delta>0.$ Then
\begin{equation}
\label{slovar}
|f^{(z)}|_{\mathcal{L},\delta}= \frac{1}{dist(z, [-2\sigma-\delta, 2\sigma +\delta])^2}. 
\end{equation}
If $ z \not\in \R, $ one has
\begin{equation}
\label{slovar1}
|f^{(z)}|_{\mathcal{L},\R}= \frac{1}{|\Im z|^2}.
\end{equation}
In a similar fashion,
for $f^{z,w}(x)= \frac{1}{z-x}- \frac{1}{w-x}, \ z,w \not\in [-2\sigma-\delta, 2\sigma +\delta].$
one has
\begin{equation}
\label{slovar2}
|f^{(z,w)}|_{\mathcal{L},\delta} \leq 2\* \frac{|z-w|}{\min\left(dist(z, [-2\sigma-\delta, 2\sigma +\delta]),dist(w, [-2\sigma-\delta, 2\sigma +\delta])
\right)^3}. 
\end{equation}
For $ z,w \not\in \R, $ one has
\begin{equation}
\label{slovar3}
|f^{(z,w)}|_{\mathcal{L},\R} \leq 2\* \frac{|z-w|}{\min(|\Im z|, |\Im w|)^3}. 
\end{equation}
\end{remark}

\begin{remark}
Applying the Poincar\'e inequality to $R_{ij}(z)$ one can replace
the estimate (\ref{odinnadtsat102}) by
\begin{equation*}
\V(R_{ij}(z))=O\left(\frac{1}{|\Im z|^4\*N}\right).
\end{equation*}
\end{remark}


\section{\bf{Fluctuations of the Resolvent Entries}}
\label{sec:flucmat}

In this section, we prove Theorems \ref{thm:resreal}, \ref{thm:resreal1}, \ref{thm:resherm}, and \ref{thm:resherm1}.
We start with the proof of Theorem \ref{thm:resreal}.
\begin{proof}[Proof of Theorem  \ref{thm:resreal}]

As in Section \ref{sec:intro}, we denote by $X^{(m)}, \ W^{(m)},$ and $R^{(m)},$  the $m\times m$ upper-left corner submatrix 
of matrices $X_N, W_N,$ and $R_N, $ where $m$ is a fixed positive integer.  We denote by 
$\tilde{X}^{(N-m)}$ the $(N-m)\times(N-m) $ lower-right corner submatrix of $X_N, $ and by 
$$\tilde{R}(z)= \left(z\* I_{N-m} - \tilde{X}^{(N-m)}\right)^{-1}, \ $$
the resolvent of $\tilde{X}^{(N-m)}. \ $  We will often drop the dependence on $z$ in the notation of 
$\tilde{R}=\tilde{R}(z) $ if it does not lead to ambiguity. In addition, 
let us denote by $x^{(1)}, \ldots, x^{(m)} \in \R^{N-m} $ the vectors such that
the components of $x^{(i)}, \ 1\leq i \leq m, $  are given by the last $N-m$ entries of the $i$-th column of the matrix $X_N. \ $  Finally,
we will denote by $B$ the $(N-m) \times m \ $ submatrix of $X_N \ $ formed by the vectors (columns) $x^{(i)}, \ 1\leq i \leq m, \ $ and by $B^*$
its adjoint matrix.

Since the fourth moment of $\mu $ and the second moment of $\mu_1$ are finite, $\|X_N\|$ converges to $2\sigma$ almost surely (\cite{BYin}).
Thus, for fixed  $z \in \C \setminus [-2\sigma, 2\sigma], $ 
$\tilde{R}=(z\* I_{N-m} - \tilde{X}^{(N-m)})^{-1} $
exists with probability 1 for all but finitely many $N$ (obviously, $\tilde{R}$ always exists for $\Im z \not=0$).
Moreover, the $m\times m$ upper-left corner
of the resolvent matrix $R_N(z)= (z\*I_N-X_N)^{-1}, \ $ denoted by $R^{(m)}(z), \ $ can be written as
\begin{equation}
\label{ugol}
R^{(m)}(z) = \left (z\*I_m - X^{(m)} - B^*\* \tilde{R}\* B \right)^{-1}=\left (z\*I_m - \frac{1}{\sqrt{N}}\*W^{(m)} - B^*\* \tilde{R}\* B \right)^{-1} .
\end{equation}
Let us denote 
\begin{equation}
\label{dvaugla}
T:=z\*I_m - \frac{1}{\sqrt{N}}\*W^{(m)} - B^*\* \tilde{R}\* B, \ \text{so} \ R^{(m)}= T^{-1}.
\end{equation}
Write
\begin{equation}
\label{hapoel}
T= \left(z-\sigma^2\*g_{\sigma}(z) \right) \* I_m -\frac{1}{\sqrt{N}} \* \Gamma_N = 
\frac{1}{g_{\sigma}(z)}\* I_m -\frac{1}{\sqrt{N}} \* \Gamma_N(z) ,
\end{equation}
where 
\begin{equation}
\label{gammamatrix}
(\Gamma_N)_{ij}(z)=
\Gamma_{ij}(z)= W_{ij} + \sqrt{N} \* \left(\langle x^{(i)}, \tilde{R} \* x^{(j)} \rangle - \sigma^2\*g_{\sigma}(z)\*\delta_{ij} \right),
\ 1\leq i,j \leq m.
\end{equation}
We rewrite (\ref{gammamatrix}) as
\begin{equation}
\label{chetyreugla}
\Gamma_N(z)=W^{(m)} + Y_N(z),
\end{equation}
where the entries of the matrix $Y_N(z)$ are given by
\begin{equation}
\label{mnogouglov}
(Y_N(z))_{ij}= Y_{ij}(z)= \sqrt{N} \* \left(\langle x^{(i)}, \tilde{R}(z) \* x^{(j)} \rangle - 
\sigma^2\*g_{\sigma}(z)\*\delta_{ij} \right),\ 1\leq i,j \leq m.
\end{equation}
\begin{remark}
The Central Limit Theorem for random sesquilinear forms (see below) implies that the entries of $Y_N(z),$ and thus the entries of $\Gamma_N(z) $ as well, 
are bounded in probability.  Recall 
that a sequence $\{\xi_N\}_{N\geq 1} $ of $\R^M$-dimensional random vectors is bounded in probability if for any $\varepsilon>0 $ there 
exists $L(\varepsilon)$ that does not depend on $N$ such that $\P(|\xi_N|>L(\varepsilon))<\varepsilon $ for all $N\geq 1.$

Then,
\begin{equation}
\label{triugla}
\sqrt{N} \*\left(R^{(m)} - g_{\sigma}(z)\*I_m \right)= g_{\sigma}^2(z) \* \Gamma_N(z) + O\left(\frac{1}{\sqrt{N}}\right),
\end{equation}
in probability (i.e. the error term multiplied by $\sqrt{N}$ is bounded in probability). 
\end{remark}

Taking into account (\ref{gammamatrix}), (\ref{triugla}), (\ref{chetyreugla}), and (\ref{mnogouglov}), we can prove the weak convergence of the 
finite-dimensional distributions of
$\Upsilon_N(z), \ z \in \C \setminus [-2\*\sigma, 2\*\sigma],$ defined in (\ref{ups}), to the finite-dimensional distributions of 
$\Upsilon(z),\ $ defined in (\ref{functionalconv}) by proving the weak convergence of the finite-dimensional 
distributions of $Y_N(z) \ $ to those of $Y(z),  \ $ defined by (\ref{dispersii1}-\ref{dispersii6}). 

To this end, we 
use the Central Limit Theorem for random sesquilinear forms due to Bai and Yao \cite{BY} in the form given by Benaych-Georges, Guionnet, and Maida in 
Theorem 6.4 in \cite{BGM}.  For the convenience of the reader, we give the formulation of this theorem in the Appendix (Theorem \ref{thm:bggm}).

Let $p $ be a fixed positive integer and 
$z_1, \ldots, z_p \in \C \setminus [-2\*\sigma, 2\*\sigma]. $  To study the joint distribution of the entries
$ ((R_N)(z_l))_{i_l,j_l}, \ 1\leq l\leq p, \ 1\leq i_l\leq j_l\leq m, $ it is enough to study  the distribution of their linear combination.
Let $ a_{s,t}^{(i)}, \ b_{s,t}^{(i)},  \ 1\leq s \leq t \leq m, \ 1\leq i \leq p, \ $  be arbitrary real numbers and consider
\begin{equation}
\label{summamatric}
\mathcal{M}_N^{(s,t)}= \sum_{i=1}^p \left( a_{s,t}^{(i)}\* \Re( \tilde{R}(z_i)) + b_{s,t}^{(i)} \* \Im( \tilde{R}(z_i))\right), \ \ 1\leq s \leq t \leq m,
\end{equation}
where for any linear operator $A$
\begin{align*}
& \Re(A)=\frac{A+A^*}{2}, \\
& \Im(A)=\frac{A-A^*}{2\*i}.
\end{align*}
Now, we show that the results of Propositions \ref{proposition:prop1} and \ref{proposition:prop2} and the almost sure convergence of 
$\|X_N\|$ to $2\sigma $ imply that
the conditions (\ref{nupogodi1}, \ref{nupogodi100}) of Theorem \ref{thm:bggm} are satisfied.  
First, we note that as $N \to \infty, $
\begin{align}
\label{91}
& \tr_N \left(\Re( \tilde{R}(z)) \* \Re( \tilde{R}(w)) \right) \to  \varphi_{++}(z,w),\\
\label{92}
& \tr_N \left(\Im( \tilde{R}(z)) \* \Im( \tilde{R}(w)) \right) \to  \varphi_{--}(z,w),\\
\label{93}
& \tr_N \left(\Re( \tilde{R}(z) \*\Im( \tilde{R}(w)\right)  \to  \varphi_{+-}(z,w),  \\
\label{94}
& \frac{1}{N} \sum_{i=1}^N (\Re( \tilde{R}(z)))_{ii} \* (\Re( \tilde{R}(w)))_{ii} \to  \Re(g_{\sigma}(z)) \*\Re(g_{\sigma}(w)),  \\
\label{95}
& \frac{1}{N} \sum_{i=1}^N (\Im( \tilde{R}(z)))_{ii} \* (\Im( \tilde{R}(z)))_{ii} \to \Im(g_{\sigma}(z)) \* \Im(g_{\sigma}(w)), \\
\label{96}
& \frac{1}{N} \sum_{i=1}^N (\Re( \tilde{R}(z)))_{ii}\*(\Im( \tilde{R}(w)))_{ii} \to \Re(g_{\sigma}(z))\*\Im(g_{\sigma}(w)),
\end{align}
for $z,w \in \C \setminus [-2\sigma, 2\sigma], $ 
where $\varphi_{++}(z,w), \varphi_{--}(z,w),$ and $\varphi_{+-}(z,w)$ are defined in (\ref{padova2}-\ref{padova4}), and
the convergence is in probability.

Indeed,  (\ref{91}-\ref{93}) follow from the semicircle law (and the convergence can be taken to be almost sure). 
In particular, for real $z$ and $w,$ in order to avoid singularities, one can replace $\tilde{R}(z), \ \tilde{R}(w)$ by
$h(X_N)\*\tilde{R}(z), \ h(X_N)\*\tilde{R}(w), $ where $h$ is defined in (\ref{hhh}), and use the fact that
$ \|X_N\|\to 2\*\sigma $ almost surely as $N \to \infty.$

Let us now prove (\ref{94}).  The proofs of (\ref{95}-\ref{96}) are
similar.  We can assume that $\Im z\not=0, \ \Im w\not=0.\ $ Otherwise, one has to replace $\tilde{R}(z)$ by $h(X_N)\*\tilde{R}(z).$
We write
\begin{align}
& |(\Re( \tilde{R}(z)))_{ii} \* (\Re( \tilde{R}(w)))_{ii}- \Re(g_{\sigma}(z)) \*\Re(g_{\sigma}(w))| \nonumber \\
& \leq |(\Re( \tilde{R}(z)))_{ii}-\Re(g_{\sigma}(z))|\*
|(\Re( \tilde{R}(w)))_{ii}| + |\Re(g_{\sigma}(z))|\* |(\Re( \tilde{R}(w)))_{ii}- \Re(g_{\sigma}(w))| \nonumber \\
& \leq |(\Re( \tilde{R}(z)))_{ii}-\Re(g_{\sigma}(z))|\* \frac{1}{|\Im w|}+ \frac{1}{|\Im z|}\*|(\Re( \tilde{R}(w)))_{ii}- \Re(g_{\sigma}(w))|. \nonumber
\end{align}
Thus, it follows from (\ref{odinnadtsat102}) that
\begin{equation*}
\E |(\Re( \tilde{R}(z)))_{ii} \* (\Re( \tilde{R}(w)))_{ii}- \Re(g_{\sigma}(z)) \*\Re(g_{\sigma}(w))|\leq \left(\frac{1}{|\Im w|} +\frac{1}{|\Im z|}\right)
\* O(N^{-1/2}),
\end{equation*}
which implies (\ref{94}).

It should be noted that Theorem \ref{thm:bggm} is proven for non-random matrices $\mathcal{M}_N^{(s,t)}, \ 1\leq s \leq t \leq m. $ 
Since the convergence in probability does not imply almost sure convergence, an additional argument is in order.  Let $\mathcal{M}_N^{(s,t)}$ be 
defined as in (\ref{summamatric}), and $u^{(i)}= x^{(i)}, \ 1\leq i \leq m. $  By calculating the second moments of
$$ \left( \sqrt{N}\* \left(\langle x^{(p)}, \mathcal{M}_N^{(p,q)}\*x^{(q)} \rangle - \delta_{pq} \* 
\tr_N(\mathcal{M}_N^{(p,p)}) \right) \right), \ \ 1\leq p,q \leq m, $$
one can show that these random variables are bounded in probability.  Therefore, it is enough to prove convergence for a subsequence $N_n\to\infty.$
Since convergence in probability implies almost sure convergence for a subsequence, we can now apply Theorem \ref{thm:bggm}  directly to a subsequence.

Applying Theorem \ref{thm:bggm}, we establish the convergence of the finite-dimensional distributions of
$Y_N(z) \ $ and obtain (\ref{dispersii1}-\ref{dispersii6}).  
Theorem \ref{thm:resreal} is proven.
\end{proof}

The proof of Theorem \ref{thm:resherm} is very similar to the proof of Theorem \ref{thm:resreal}, Theorem \ref{thm:bggm}   
plays the central role in our arguments again.
We choose 
\begin{equation}
\label{summamatric1}
\mathcal{M}_N^{(s,t)}= \sum_{i=1}^p \left( a_{s,t}^{(i)}\* \Re( \tilde{R}(z_i)) + b_{s,t}^{(i)} \* \Im( \tilde{R}(z_i))\right), \ \ 1\leq s \leq t 
\leq m,
\end{equation}
where $ a_{s,t}^{(i)}, \ b_{s,t}^{(i)},  \ 1\leq s < t \leq m, \ 1\leq i \leq p, \ $  are arbitrary complex numbers and 
$a_{s,s}^{(i)}, \ b_{s,s}^{(i)},  \ 1\leq s \leq m, \ 1\leq i \leq p \ $ are arbitrary real numbers.  Applying Theorem 6.1, we establish the convergence 
of the finite-dimensional distributions of $Y_N(z)$ and obtain (\ref{dispersii11}-\ref{dispersii16}).

Now, we prove Theorem \ref{thm:resreal1}.
\begin{proof}[Proof of Theorem  \ref{thm:resreal1}]
Let $\mu$ and $\mu_1$ satisfy the Poincar\'e inequality (\ref{poin}).
To prove the functional limit theorem, i.e. the convergence in distribution of the sequence of 
probability measures $\mathcal{P}_N $ on $C(\mathcal{D},\C^{m(m+1)/2}), $  it is now sufficient to prove that the sequence $\mathcal{P}_N $ is tight
(\cite{Bil}).  For this, we need to show that for
\begin{equation*}
\Upsilon_N(z)=\sqrt{N} \* \left(R^{(m)}(z)-g_{\sigma}(z)\*I_m \right), \ z \in \C \setminus [-2\*\sigma, 2\*\sigma].
\end{equation*}
the following conditions are satisfied:

(a) for some fixed $z_0 \in \mathcal{D}, \ $ for each $\varepsilon >0, $ there exist sufficiently large $K$ and $N_0$ such that
\begin{equation}
\label{pervoeusl}
\P(\|\Upsilon_N(z_0)\|\geq K)\leq \epsilon \ \ \text{for all} \ N\geq N_0,
\end{equation}
and

(b) for each $\varepsilon >0 $ and $ \alpha>0 $ there exist $\gamma>0$ and $N_0$ such that
\begin{equation}
\label{vtoroeusl}
\P ( \omega_{\Upsilon_N}(\gamma) \geq \alpha) \leq \varepsilon \ \ \text{for all} \ N\geq N_0,
\end{equation}
where  $\omega_{\Upsilon_N}(\gamma)$ denotes the modulus of continuity of $Y_N(z)$ on $\mathcal{D}, \ $ namely
\begin{equation}
\omega_{\Upsilon_N}(\gamma)= sup_{|z-w| \leq \gamma} \|\Upsilon_N(z)-\Upsilon_N(w)\|,
\end{equation}
where the supremum is taken over all $z,w \in \mathcal{D} \ $ such that $|z-w| \leq \gamma.$

The property (a) immediately follows from the definition of $\Upsilon_N$ and the bounds (\ref{odinnadtsat100}-\ref{odinnadtsat102}) in 
Proposition \ref{proposition:prop1}.
To prove (b), we replace $R_N(z)$ by $h(X_N)\*R_N(z)$ in the definition of $\Upsilon_N(z)$ (if $\mathcal{D} \cap \R=\emptyset,$
this procedure is not needed), where $h$ is defined in  (\ref{hhh}) in such a way that $supp(h)\cap \mathcal{D}=\emptyset.$
We note that $R_N(z)=h(X_N)\*R_N(z)$ almost surely for all $z$ and for sufficiently large $N.$  

It then follows from the results of Proposition \ref{proposition:prop4} that uniformly
in $z,w \in \mathcal{D}, $ 
\begin{equation}
\label{sasa}
\P \left(|\Upsilon_N(z)-\Upsilon_N(w)|
\geq t \right) \leq 2\*K \* \exp\left( -const_1 \* \frac{t}{|z-w|} \right) + (2\*K+o(1))\* \exp\left(-\frac{\sqrt{\upsilon\*N\*\delta}}{2}\right).
\end{equation}
with some $const_1>0.$
In addition, for any $z \in \mathcal{D}, $
\begin{equation}
\label{sosa}
\P \left(| \frac{d\*(\Upsilon_N)_{ij}(z)}{dz}|
\geq t \right) \leq 2K \* \exp( -const_2 \* t) + (2\*K+o(1))\* \exp\left(-\frac{\sqrt{\upsilon\*N\*\delta}}{2}\right),
\end{equation} 
for some $const_2>0.\ $
Without loss of generality, we can assume that $\mathcal{D}$ is a rectangle with the sides parallel to the coordinate axes.  We then partition
$\mathcal{D}$ into $O(2^{2n}) $ small squares $\sqcup_i D^{(n)}_i \ $ of the diameter 
$2^{-n}, \ 1\leq n \ \leq n_1=const \* \log(N)\*(1+o(1)), \ $ where $const>0 $ 
is chosen so that
\begin{equation}
\label{ucla}
const>\frac{\log 2}{2}.
\end{equation}
We then estimate the probability of the event that 
\begin{equation}
\label{gar1}
s_n:=\sup \|\Upsilon_N(z)-\Upsilon_N(w)\| \geq A\* n^{100}\*2^{-n},
\end{equation}
where the supremum in (\ref{gar1}) 
is taken over all pairs $(z,w)$ that are the vertices of the same small square.   Using (\ref{sasa}), one can show that 
this probability is 
\begin{equation*}
O\left(\exp(-\frac{const_1\*A}{2}\*n^{100})\right) +O\left(\exp(-\frac{\sqrt{\upsilon\*N\*\delta}}{2})\right)
\end{equation*}
uniformly in $N.\ $ 
We can also estimate the probability of the event
\begin{equation}
\label{gar2}
S_{n_1}:=\sup \| \frac{d\*(\Upsilon_N)(z)}{dz} \|  \geq A\* n_1^{100}= A\* const^{100}\* (\log(N))^{100}\*(1+o(1)),
\end{equation}
where the supremum in (\ref{gar2}) is taken over all vertices of the partition $\sqcup_i D^{(n_1)}_i , \ $ by 
\begin{equation*}
O\left(\exp(-\frac{const_2\*A}{2}\*n_1^{100})\right) +O\left(\exp(-\frac{\sqrt{\upsilon\*N\*\delta}}{2})\right),
\end{equation*}
uniformly in $N.$

Finally, we note that since we choose $const>0 $ in  $n_1=const \* \log(N)\*(1+o(1))$ to be sufficiently large so that
(\ref{ucla}) is satisfied, we have
\begin{equation}
\label{gar3}
|S_{n_1}- \sup_{z \in\mathcal{D}}  \| \frac{d\*(\Upsilon_N)(z)}{dz} \|| \leq 1,
\end{equation}
since  the second derivatives of the entries of $\Upsilon_N(z) $ are trivially bounded by \\
$const_3\*\sqrt{N},$
where $const_3$ depends on $\mathcal{D}.$
Now choosing $A$ sufficiently large, we can make the probability 
$ \P\left(\omega_{Y_N}(c\* \alpha \* |\log\alpha|) \geq \alpha \right) \ $
smaller than $\epsilon $ for a suitable constant $c>0. \ $
We leave the details to the reader.
\end{proof}


\section {\bf Fluctuations of Matrix Entries of Regular Functions of Wigner Matrices}
\label{sec:lipschitz}
We give the proofs in the real symmetric case (Theorems \ref{thm:real} and \ref{thm:real1}).  
The proofs in the Hermitian case (Theorems \ref{thm:herm} and \ref{thm:herm1}) are very similar.
First, we assume that $\mu$ and $\mu_1$ satisfy the Poincar\'e inequality (\ref{poin}) and prove
Theorem \ref{thm:real1}.
Then we will extend it to the case of finite fourth moment and prove Theorem \ref{thm:real}.
\begin{proof}[Proof of Theorem  \ref{thm:real1}]
We start by considering a test function $f$ which is analytic in a neighborhood of $[-2\*\sigma, 2\*\sigma]\ $ and takes real values on $\R.\ $
We write
\begin{equation}
\label{integral107}
f(X_N)_{ij}= \frac{1}{2\*\pi\* i} \* \int_{\gamma} f(z) \* (R_N(z))_{ij} \* dz,
\end{equation}
where $\gamma$ is a clockwise-oriented contour in the domain of analyticity of $f$ that encircles the interval 
$[-2\*\sigma, 2\*\sigma].\ $  Then
\begin{equation}
\label{countint}
\sqrt{N}\*\left(f(X_N)_{ij}-\delta_{ij}\* \int_{-2\sigma}^{2\sigma} f(x) \* \frac{1}{2 \pi \sigma^2} \sqrt{ 4 \sigma^2 - x^2}  \* dx \right)
= \frac{1}{2\*\pi\* i} \* \int_{\gamma} f(z)\* (\Upsilon_N(z))_{ij} \* dz, \ 1\leq i \leq j \leq m.
\end{equation}
By the functional convergence in Theorem \ref{thm:resreal1}, the r.h.s. in (\ref{countint}) converges in distribution to the distribution of
independent (up from the diagonal) random variables 
\begin{equation}
\label{countint1}
\frac{1}{2\*\pi\* i} \* \int_{\gamma} f(z)\* \Upsilon_{ij}(z) \* dz= 
\left(\frac{1}{2\*\pi\* i} \* \int_{\gamma} f(z)\*g_{\sigma}^2(z) \* dz\right) \* W_{ij}+
\frac{1}{2\*\pi\* i} \* \int_{\gamma} f(z)\* \*g_{\sigma}^2(z)\*Y_{ij}(z) \* dz, 
\end{equation}
$ 1\leq i \leq j \leq m.$
We evaluate
\begin{align}
\label{wignerterm}
& \frac{1}{2\*\pi\* i} \* \int_{\gamma} f(z)\*g_{\sigma}^2(z) \* dz= \frac{1}{2\*\pi\* i} \* \int_{\gamma} f(z) 
\* \frac{z\*g_{\sigma}(z)-1}{\sigma^2} \* dz= \\
& \frac{1}{\sigma^2}\*\frac{1}{2\*\pi\* i} \* \int_{\gamma} f(z)\*z\*g_{\sigma}(z) \* dz= \frac{1}{\sigma^2} \* 
\int_{-2\sigma}^{2\sigma} x\*f(x) \* \frac{1}{2 \pi \sigma^2} \sqrt{ 4 \sigma^2 - x^2}  \* dx. \nonumber
\end{align}
The last integral in (\ref{countint1}) is a real Gaussian random vector with independent entries and variances
\begin{equation}
\frac{1}{2\*\pi\* i} \* \int_{\gamma} \frac{1}{2\*\pi\* i} \* \int_{\gamma} f(z)\* f(w) \* \left(\E(g_{\sigma}^2(z)\*g_{\sigma}^2(w)
\*Y_{ij}(z)\*Y_{ij}(w))-
\E(g_{\sigma}^2(z)\*Y_{ij}(z))\*
\E(g_{\sigma}^2(w)\*Y_{ij}(w)) \right) \* dx\* dw.
\end{equation}
Let us first consider the off-diagonal case $i\not=j. \ $
By (\ref{dispersii4}-\ref{dispersii6}) and (\ref{padova1}-\ref{padova4}),
we have 
\begin{align}
& \E\left(g_{\sigma}^2(z)\*g_{\sigma}^2(w)\*Y_{ij}(z)\*Y_{ij}(w)\right)- \E\left(g_{\sigma}^2(z)\*Y_{ij}(z)\right)
\*\E\left(g_{\sigma}^2(w)\*Y_{ij}(w)\right)
= \sigma^4\*g_{\sigma}^2(z)\*g_{\sigma}^2(w)\* \varphi(z,w) \nonumber \\
& =-\sigma^4\*g_{\sigma}^2(z)\*g_{\sigma}^2(w) \frac{g_{\sigma}(z)-g_{\sigma}(w)}{z-w}= -\sigma^4\*g_{\sigma}(z)\*g_{\sigma}(w) \*
\frac{g_{\sigma}^2(z)\*g_{\sigma}(w)- g_{\sigma}(z)\*g_{\sigma}^2(w)}{z-w} \nonumber \\
& =-\sigma^2\*g_{\sigma}(z)\*g_{\sigma}(w) \*\frac{z\*g_{\sigma}(z)\*g_{\sigma}(w)-g_{\sigma}(w)-w\*g_{\sigma}(z)\*g_{\sigma}(w) +g_{\sigma}(z)}{z-w}
\nonumber \\
& =-\sigma^2\*g_{\sigma}(z)\*g_{\sigma}(w) \* \left( g_{\sigma}(z)\*g_{\sigma}(w) + \frac{g_{\sigma}(z)-g_{\sigma}(w)}{z-w} \right) \nonumber \\
& = -\sigma^2\*g_{\sigma}^2(z)\*g_{\sigma}^2(w) 
+ \left(\varphi(z,w) - g_{\sigma}(z)\*g_{\sigma}(w)\right). \nonumber
\end{align}
We note that $ \varphi(z,w) - g_{\sigma}(z)\*g_{\sigma}(w)=\Cov(\frac{1}{z-\eta},\frac{1}{w-\eta}), \ $ where $\eta \ $ is  distributed according to the 
semicircle law (\ref{polukrug}).  Then
\begin{align}
& \frac{1}{2\*\pi\* i} \* \int_{\gamma} \frac{1}{2\*\pi\* i} \* \int_{\gamma} f(z)\* f(w) \* (\varphi(z,w) - g_{\sigma}(z)\*g_{\sigma}(w))\*dz\*dw= 
\nonumber \\
& \frac{1}{2\*\pi\* i} \* \int_{\gamma} \frac{1}{2\*\pi\* i} \* \int_{\gamma} f(z)\* f(w) \* \Cov(\frac{1}{z-\eta},\frac{1}{w-\eta})\*dz\*dw=
\V(f(\eta)). \nonumber
\end{align}
This together with (\ref{wignerterm}),(\ref{countint1}) proves (\ref{trudno2}), (\ref{dsqf}) for analytic functions.

In the diagonal case $i=j, $ the previously studied terms contribute to
\begin{equation*}
\E(g_{\sigma}^2(z)\*g_{\sigma}^2(w)\*Y_{ij}(z)\*Y_{ij}(w))- \E(g_{\sigma}(z)\*Y_{ij}(z))\*\E(g_{\sigma}(w)\*Y_{ij}(w))
\end{equation*}
with a factor of two.  In addition, (\ref{dispersii1}-\ref{dispersii3}) provide
one more term \\
$ \ \kappa_4(\mu)\* g_{\sigma}^3(z)\*g_{\sigma}^3(w). \ $
Evaluating
\begin{align}
\label{wignerterm100}
& \frac{1}{2\*\pi\* i} \* \int_{\gamma} f(z)\*g_{\sigma}^3(z) \* dz= \frac{1}{2\*\pi\* i} \* \int_{\gamma} f(z) \* \frac{(z^2-\sigma^2)\*g_{\sigma}(z)-z}
{\sigma^4} \* dz= \\
& \frac{1}{\sigma^4}\*\frac{1}{2\*\pi\* i} \* \int_{\gamma} f(z)\*(z^2-\sigma^2)\*g_{\sigma}(z) \* dz= 
\frac{1}{\sigma^4} \* \int_{-2\sigma}^{2\sigma} (x^2-\sigma^2)\*f(x) \* \frac{1}{2 \pi \sigma^2} \sqrt{ 4 \sigma^2 - x^2}  \* dx, \nonumber
\end{align}
we prove (\ref{trudno1}), (\ref{trudno2}) in the analytic case with the centralizing
constant $\int_{-2\sigma}^{2\sigma} f(x) \* d \mu_{sc}(dx).$

It follows from (\ref{bol}) that if $f$ satisfies the subexponential growth condition (\ref{exprost}) on the real line then
we can choose the centralizing constants to be $\E f(X_N)_{ij} $ in (\ref{trudno1}), (\ref{trudno2}).
To extend the results of Theorem \ref{thm:real} to a more general class of functions we apply a standard approximation procedure.
If $f$ is Lipschitz continuous on $[-2\*\sigma-\delta, 2\*\sigma +\delta]$  and satisfies (\ref{exprost}), we choose a sequence of analytic functions
$\{f_n\}, \ n\geq 1, $ such that
\begin{equation*}
f_n(0)=f(0), \ n\geq 1, \ \text{and} \ |f_n-f|_{\mathcal{L},\delta} \to 0 \ \text{as } \ n\to \infty,
\end{equation*}
where $|f|_{\mathcal{L},\delta}$ has been defined in (\ref{snova1}).
Let us also choose $h$ in such a way that
$h:\R \to \R $ is a smooth function with compact support, $ h(x)=1 $ for $x \in [-2\*\sigma-\delta/2, 2\*\sigma +\delta/2], $ 
and $ h(x)=0 $ for $|x| \geq 2\*\sigma +\frac{3}{4}\*\delta. $ 
We observe that for any  $n\geq 1 $
\begin{equation*}
f(X_N)\not= (f\*h)(X_N), \ f_n(X_N)\not= (f_n\*h)(X_N), 
\end{equation*}
with probability exponentially small in $\sqrt{N}$.  In addition,
\begin{align}
& \E |(f\*(1-h))(X_N)_{ij}|=O(\exp(-const\*\sqrt{N})), \nonumber \\
& \E |(f_n\*(1-h))(X_N)_{ij}|=O(\exp(-const_n\*\sqrt{N})), \ n\geq 1, \nonumber
\end{align}
where $const>0, \ const_n>0. \ $ We then choose $n$ sufficiently large so that \\
$|f\*h-f_n\*h|_{\mathcal{L}} \leq \varepsilon. \ $ As in the proof of Proposition \ref{proposition:prop4}, we use the fact that for any 
Lipschitz continuous $f$, 
the function $f(X)_{ij} $ is a Lipschitz continuous function of the matrix entries of $X.$  Therefore, we can show
that  $\ \V\left(\sqrt{N}\*((f\*h)(X_N)_{ij}-(f_n\*h)(X_N)_{ij})\right) \ $ can be made arbitrary small (uniformly in $N$ ) for sufficiently large $n$ if
we apply the concentration inequality
(\ref{BOLUKLON}) to $(f\*h)(X_N)_{ij}-(f_n\*h)(X_N)_{ij}. $ Finally, we observe that
$\omega^2(f_n)\to \omega^2(f), \ \alpha^2(f_n) \to \alpha^2(f), \ \beta^2(f_n)\to \beta^2(f), \ d^2(f_n)\to d^2(f) $ as $n \to \infty.$

It follows from Proposition \ref{proposition:prop2} and (\ref{bol})
that if $f$ is seven times continuously differentiable on $[-2\*\sigma-\delta, 2\*\sigma +\delta]$  
(six times continuously differentiable on $[-2\*\sigma-\delta, 2\*\sigma +\delta]$  in the off-diagonal case $i\not=j $) and
satisfies (\ref{exprost}) then one can replace $\E f(X_N)_{ij} $ by $\delta_{ij}\* \int f(x) \* d \mu_{sc}(dx) $ in (\ref{trudno1}), (\ref{trudno2}) since
$$ \ \E f(X_N)_{ij}= \delta_{ij}\* \int f(x) \* d \mu_{sc}(dx) + O(\frac{1}{N}). $$
Theorem \ref{thm:real1} is proven.
\end{proof}

Now, we prove Theorem \ref{thm:real} assuming only that $\mu$ and $\mu_1$ have finite fourth moments.  The role of (\ref{BOLUKLON})
will be played by (\ref{chto3}). 
\begin{proof}[Proof of Theorem  \ref{thm:real}]
By Theorem \ref{thm:resreal} and Proposition \ref{proposition:prop1}, we have the result for finite linear combinations
\begin{equation}
\label{120}
f(x)=\sum_{l=1}^k a_l\* \frac{1}{z_l-x}, \ z_l \not\in [-2\sigma, 2\sigma], \ 1\leq l \leq k, 
\end{equation}
and, more generally, for 
\begin{equation}
\label{1200}
f(x)=\sum_{l=1}^k a_l\* h_l(x)\*\frac{1}{z_l-x}, \ z_l \not\in [-2\sigma, 2\sigma], \ 1\leq l \leq k, 
\end{equation}
where $h_l\in C^{\infty}_c(\R), \ 1\leq l\leq k,$ satisfy (\ref{hhh}).
Applying the Stone-Weierstrass theorem (see e.g \cite{RS}), one can show that such functions are dense in $C^4_c(\R).$  
Therefore, we can approximate an arbitrary $f\in C^4_c(\R)$ by 
such functions $h(x)\*f_n(x)$ in such a way that  
\begin{equation*}
supp (f)\subset [-A, A], \ supp(hf_n) \subset [-A, A], 
\end{equation*}
for all $n$ and sufficiently large $A>0,$
and $\|f-hf_n\|_{C^4([-A,A])} \to 0, \ $ as $n \to \infty.$
It follows from (\ref{chto3}) that $ \V(f(X_N)_{ij}- (h(x)\*f_n)(X_N)_{ij})  $ can be made arbitrary small uniformly in $n$ provided we choose $n$ to be 
sufficiently large.  Since $\omega^2(hf_n)\to \omega^2(f), \ \alpha^2(hf_n) \to \alpha^2(f), \ \beta^2(hf_n)\to \beta^2(f), \ d^2(hf_n)\to d^2(f) $ as 
$n \to \infty,$  the result follows.  
Theorem \ref{thm:real} is proven.
\end{proof}


\section{\bf{Appendix}}
\label{section:appendix}

In our analysis, we need to study the expectation of the random matrix entries multiplied by functions of the random matrix. In order to handle this we 
use the following decoupling formula \cite{KKP}:
 
Given $\xi$, a real-valued random variable with $p+2$ finite moments, and $\phi$ a function from $\C \to \R$ with $p+1$ continuous and bounded 
derivatives then: 
\begin{equation}
\label{decouple} 
 \E(\xi \phi(\xi)) = \sum_{a=0}^p \frac{\kappa_{a+1}}{a!} \E(\phi^{(a)}(\xi)) + \epsilon  
 \end{equation}
Where $\kappa_a$ are the cumulants of $\xi$, $|\epsilon| \leq C \sup_t \big| \phi^{(p+1)}(t) \big| \E(|\xi|^{p+2})$, $C$ depends only on $p$. \\

For any two Hermitian matrices $X_1$ and $X_2$ and non-real $z$ we have the resolvent identity:
\begin{equation}
\label{resident}
 (zI - X_2)^{-1} = (zI - X_1)^{-1} - (zI - X_1)^{-1}(X_1 - X_2) (zI - X_2)^{-1}  
\end{equation}

If $X$ is a real symmetric matrix with resolvent $R$ then the derivative of $R_{kl}$ with respect to $X_{pq}$, for $p \not = q$ is given by 
\begin{equation}
\label{vecher1}
\frac{\partial R_{kl}}{\partial X_{pq}} = R_{kp}\*R_{ql} +R_{kq} \*R_{pl}.
\end{equation}
If $p =q$ then the derivative is:
\begin{equation}
\label{vecher2}
\frac{\partial R_{kl}}{\partial X_{pp}} = R_{kp}\*R_{pl} .
\end{equation}

In a similar way, if $X$ is a Hermitian matrix then the derivative of $R_{kl}$ with respect to $\Re X_{pq}, \ \Im X_{pq}$, for $p \not = q$ are given by
\begin{align}
& \frac{\partial R_{kl}}{\partial \Re X_{pq}} = R_{kp}\*R_{ql} +R_{kq} \*R_{pl}, \\
& \frac{\partial R_{kl}}{\partial \Im X_{pq}} = i\*\left(R_{kp}\*R_{ql} - R_{kq}\*R_{pl} \right).
\end{align}
When $p=q$ then
\begin{equation}
\frac{\partial R_{kl}}{\partial X_{pp}} = R_{kp}\*R_{pl} .
\end{equation}

We will use the following bounds on the resolvent:
\begin{equation}
\label{norma17}
\| R_N(z) \|=\frac{1}{dist(z, Sp(X))},
\end{equation}
where by $Sp(X) $ we denote the spectrum of a real symmetric (Hermitian) matrix $X.$
(\ref{norma17}) implies
\begin{equation}
\label{resbound}
\| R_N(z) \| \leq |\Im(z)|^{-1} 
\end{equation}
which also implies all the entries of the resolvent are bounded by $|\Im(z)|^{-1}$.
Similarly, we have the following bound for the Stieltjes transform, $g(z)$, of any probability measure:
\begin{equation}
\label{STbound}
 | g(z) | \leq |\Im(z)|^{-1} 
\end{equation}

Below, we state the Central Limit Theorem for random sesquilinear forms of Bai and Yao in the form given in \cite{BGM}. 

\begin{theorem}
\label{thm:bggm}
Let us fix $m\geq 1$ and let, for each $N, \ \mathcal{M}_N^{(s,t)}, \ 1\leq s,t \leq m, $ be a family of 
$N\times N$ real (resp. complex) matrices such that for all
$s,t, \ \mathcal{M}_N^{(t,s)}= \left(\mathcal{M}_N^{(s,t)}\right)^* $ and such that for all $s,t=1,\ldots, m,$
\begin{align}
\label{nupogodi1}
& \tr_N \left(\mathcal{M}_N^{(s,t)} \* \mathcal{M}_N^{(t,s)}\right)\to \sigma^2_{s,t}, \ \ \text{as} \ n\to \infty, \\
\label{nupogodi100}
& \frac{1}{N}\* \sum_{i=1}^N |(\mathcal{M}_N^{(s,s)})_{ii}|^2 \to \gamma_s, \ \ \text{as} \ n\to \infty.
\end{align}
Let $u^{(1)}, \ldots, u^{(m)}$ be a sequence of i.i.d. random vectors in $\R^N$ (resp. $\C^N$) such that the $N$ coordinates of $u^{(1)}$ 
are i.i.d. centered
real (resp. complex) centered random variables distributed according to a probability measure with variance one and finite fourth moment.  
In the complex case, we also assume that real and imaginary 
parts of each coordinate of $u^{(1)}$ are independent and identically distributed according to a probability measure $\nu $ on the real line.

For each $N,$ define the $m\times m$ random matrix
\begin{equation}
\label{avstral}
G_N:=\left( \sqrt{N}\* \left(\langle u^{(p)}, \mathcal{M}_N^{(p,q)}\*u^{(q)} \rangle - \delta_{pq} \* 
\tr_N(\mathcal{M}_N^{(p,p)}) \right) \right), \ \ 1\leq p,q \leq m.
\end{equation}
Then the distribution of $G_N$ converges weakly to the distribution of a real symmetric (resp. Hermitian) random matrix 
$G=(g_{p,q}), \ 1\leq p,q \leq m, $
such that the random variables
$$ \{g_{p,q}, \ 1\leq p,q \leq m \},  \ \text{(resp.} \{g_{s,s}, \ 1\leq s \leq m, \ \Re g_{p,q}, \ \Im g_{p,q}, \ 1\leq p,q \leq m \} )$$
are independent for all $s, \ g_{ss}\sim N(0, 2\*\sigma^2_{s,s}+ \kappa_4(\nu)\*\gamma_s) \ (\text{resp.} $ \\
$ g_{ss}\sim N(0, \sigma^2_{s,s}+ \frac{1}{2}\*\kappa_4(\nu)\*\gamma_s)), \ $ and for all $p\not=q, \ g_{s,t}\sim N(0, \sigma^2_{p,q}), $\\
(resp.) $ \ \Re g_{s,t}\sim N(0, \frac{1}{2}\*\sigma^2_{p,q}),\ \Im g_{s,t}\sim N(0, \frac{1}{2}\*\sigma^2_{p,q}),$
where $\kappa_4(\nu)$ denotes the fourth cumulant of $\nu.$
\end{theorem}

\begin{remark}
Almost simultaneously with our paper, L.Pastur and A. Lytova posted a preprint \cite{PL10} where they extended the technique of \cite{LP} and gave 
another proof the convergence in distribution for a normalized diagonal entry $\sqrt{N}\*(f(X_N)_{11}-\E (f(X_N)_{11})$ under the conditions that a 
real symmetric Wigner matrix $X_N$ has i.i.d. entries up from the diagonal and the cumulant generating function $\log (\E e^{z\*W_{12}})$ is entire. 
Pastur and Lytova require that a test function $f$ satisfies
$$ \int_{\R} (1+|k|)^3 \* |\hat{f}(k)|\* dk <\infty,$$
where $\hat{f}(k)$ is the Fourier transform of $f.$
\end{remark}

\end{document}